\documentclass[a4paper,12pt,twoside]{article}
\usepackage[latin1]{inputenc}
\usepackage[dvips]{graphics}
\usepackage{setspace}
\input xy
\input cyracc.def

\xyoption{all}
\usepackage{xr}
\usepackage{amssymb}
\usepackage{amsmath}
\usepackage{amsthm}
\usepackage{amsfonts}
\usepackage{mathrsfs}
\theoremstyle{plain}
\newtheorem{T}{Theorem}[section]

\newtheorem{Cor}[T]{Corollary}
\newtheorem{TL}[T]{Lemma}
\newtheorem{Prop}[T]{Proposition}

\newcommand\mps[1]{}

\theoremstyle{definition}\newtheorem{D}{Definition}[T]
\theoremstyle{definition}
\theoremstyle{remark}
\theoremstyle{remark}\newtheorem{remark}{Remark}[T]
\newcommand{\ie}{\hbox{\textit{i.e.} }}

\newcommand{\loccit}{\hbox{\textit{loc. cit.} }}

\newcommand{\Aut}{\hbox{Aut}}
\newcommand{\End}[1]{\hbox{End}#1}

\newcommand{\bb}[1]{\mathbb{#1}}
\newcommand\calo{\mathcal O}

\newcommand\ibid{\emph{ibid}}
\newcommand{\mc}[1]{\mathcal{#1}}

\newcommand{\lci}{local complete intersection }

\newcommand{\Hom}{\mathrm{Hom}}
\newcommand{\holim}{\operatorname {holim}}

\newcommand{\fhom}{\mathbf{hom}}
\newcommand{\uHom}{\underline {Hom}}

\newcommand\op[1]{\operatorname{#1}}

\newcommand\rk{\operatorname{rk}}

\newcommand\Pic{\operatorname{Pic}}
\newcommand\spec{\operatorname{Spec}}

\input cyracc.def

\theoremstyle{plain}

\theoremstyle{definition}

\newcommand{\ses}[5]{$$
\xymatrix@1{ 0 \ar[r] & {{#1}} \ar[r]^-{{#2}} & {{#3}} \ar[r]^-{{#4}} &
{{#5}} \ar[r] & 0
\\ }$$}

\newcommand{\sesdot}[5]{$$
\xymatrix@1{ 0 \ar[r] & {{#1}} \ar[r]^-{{#2}} & {{#3}} \ar[r]^-{{#4}} &
{{#5}} \ar[r] & 0. \\ }$$}

\newcommand{\sesbig}[5]{$$\xymatrix@1{{\raisebox{1.0ex}[3.0ex][1.0ex]{$0$}}
\ar@<0.6ex>[r] & {\raisebox{1.0ex}[3.0ex][1.0ex]{${#1}$}}
\ar@<0.6ex>[r]^{#2} & {\raisebox{1.0ex}[3.0ex][1.0ex]{${#3}$}}
\ar@<0.6ex>[r]^{#4} & {\raisebox{1.0ex}[3.0ex][1.0ex]{${#5}$}}
\ar@<0.6ex>[r] & {\raisebox{1.0ex}[3.0ex][1.0ex]{$0$}}}$$}

\newcommand{\Comment}[1]{}

\title{A Deligne-Riemann-Roch isomorphism I: \\ Preliminaries on virtual categories}
\author{Dennis Eriksson}
\begin{document}
\maketitle
\tableofcontents
\section{Introduction}

This is the first article in an upcoming series of papers. They have arisen through an attempt to answer open questions of Deligne proposed in \cite{determinant}. It amounts to functorial and metrized versions of a Grothendieck-Riemann-Roch theorem, as well as a Lefschetz-Riemann-Roch formula in the sense of Thomason, cf. \cite{thomason1}, Theorem 3.5.  \\

This is supposed to be understood in the following sense (for which we refer to loc.cit. for the best introduction). The Grothendieck-Riemann-Roch theorem is the formula
$$\op{ch}(Rf_* E) = f_* (\op{ch}(E) \op{Td}(\Omega_f)^{-1})$$
for an \lci morphism of schemes $f: X \to Y$ (cf. \cite{Riemann-RochAlgebra}, chapter V, \S 7 or the book \cite{SGA6} for a precise formulation). The general question on functoriality becomes whether there are categorical replacements of all the objects and homomorphisms involved. This is an approach to obtain secondary information which gets lost when one quotients out with various equivalences. \\

Some remarks are in order. The approach using virtual categories as we have done are, as was already noted in \loccit, most suitable to treat questions on the determinant of the cohomology. This can for example be seen from the observation that the virtual category is a truncation of the $K$-theory space and the determinant of the cohomology corresponds to weight one Adams eigenspaces which have all their $K$-theory in $K_0$ and $K_1$. A generalization of these results to encompass higher order information seems to be within reach, but the work has not been effectuated. \\

All these results are with rational coefficients whereas the original result of Deligne in the case of smooth curves was integral up to universal sign. The empire of signs one would have to tackle escape my understanding for the moment. However, the signs in question do not seem to have intrinsic significance and to my best knowledge the applications don't seem to mind. It would however be desirable to have an integral version of the functorial Riemann-Roch theorem (the one with rational coefficients is to be dealt with in an upcoming paper) as obtained by Deligne for smooth curves. This does seem to have intrinsic significance. \\

The article is organized as follows. We first consider determinant functors on Waldhausen categories and triangulated categories establishing results on associated virtual categories (cf. Theorem \ref{thm: waldhausen-determinant} and Proposition \ref{prop:determinanttriangulated}. This extends the situation of determinant categories on exact categories in \cite{determinant}. In the next section we consider various natural operations on virtual categories and in particular construct Adams and $\lambda$-operations (cf. Proposition \ref{Prop:Adamsoperation}, Proposition \ref{prop:lambdasoule}). We then establish certain rigidity results on operations on virtual categories in Theorem \ref{thm:rigidity}, giving existence of lifts of operations on $K_0$ to the virtual category, at least after inverting the integers. Finally in the last section we make comparisons to constructions of Franke in "Chern functors" in \cite{AAG} of Chern intersection functors. We have also included some two appendices on $\bb A^1$-homotopy theory and algebraic stacks to fix language and recall the most necessary utensils. \\

Finally, it should be noted that the results in this article bear a striking resemblance to the results in \cite{Gillet-Riemann-Roch}. The methods are very similar in that they use homotopy theory simplicial sheaves to tackle Riemann-Roch-type problems. Also related results have been obtained in a series of papers \cite{ChowCT}, "Chern functors" in \cite{AAG} and \cite{Franke-unpublished}. \\

\textbf{Acknowledgements}: I'd like to thank Damian Rössler who was my advisor during the writing of this series of articles, in particular for suggesting this area of research as a thesis topic. I'd also like to thank Takeshi Saito, Shinichi Mochizuki and Joel Riou for various exchanges on the topics at hand.
\section{The virtual category and triangulated categories}

Given a small exact category $\mc C$, we can consider its $K$-theory. The first case of $K_0$ can be defined explicitly in terms of the category $\mc C$, as the Grothendieck group of $\mc C$. This is the free abelian group on the objects of $\mc C$, modulo the relationship $B = A + C$ if $$0 \to A \to B \to C \to 0$$ is an exact sequence in $\mc C$. A more sophisticated approach was taken by Quillen, \cite{Quillen}, where he constructs a certain topological space $BQ \mc C$ associated to a (small) exact category $\mc C$ such that $$K_i(\mc C) := \pi_{i+1}(BQ \mc C).$$ 
Now, let $X \in \op{ob(Top_\bullet)}$ be a pointed topological space. One defines the fundamental groupoid of $X$ to be the category whose objects are points of $X$, and morphisms are homotopy-classes of paths, \ie it is associated to the diagram $$[PX \rightrightarrows X]$$ where $PX$ is the space of paths of $X$. Denote the corresponding functor by $$\pi_f: \op{Top_\bullet} \to \op{Grp}.$$ Deligne, \cite{determinant} defines a category of virtual objects of an exact category, which offers a type of abelianization of the derived category and $K_0$ of the category. \\
Let $\mathcal C$ be a small exact category. The category of virtual objects of $\mc C$, $V(\mathcal C)$ is the following: Objects are loops in $BQ \mc C$ around a fixed zero-point, and morphisms are homotopy-classes of homotopies of loops. Recall that $BQ \mc C$ is the geometrical realization of the Quillen $Q$-construction of $\mc C$. Addition is the usual addition of loops. This construction is the fundamental groupoid of the space $\Omega BQ\mc C$. In case $\mc C$ is not small we will always consider an equivalent small category, and ignore any purely categorical issues this might cause.
\begin{remark} $V(\mathcal C)$ is a groupoid, \ie any morphism is an isomorphism, and the set of equivalence-classes is in natural bijection with $K_0(\mc C)$. For any object $c \in ob V(\mathcal C)$, we have $\Aut_{V(\mathcal C)}(c) = \pi_1(\Omega BQ\mathcal C) = K_1(\mc C)$.
\end{remark}
Deligne also provides a more algebraic and universal definition of $V(\mathcal C)$. We will give an additional description in terms of derived categories in the next section (cf. Theorem \ref{thm: waldhausen-determinant} for a precise statement). \\

\subsection{Algebraic definition} \label{picardcategory-definition} The above category is a so called universal Picard category with respect to $\mathcal C$. A (commutative) Picard category is a groupoid $\mathcal C$ with an auto-equivalence $P \mapsto P \oplus Q$ for any object $Q$ of $\mathcal C$, satisfying certain compatibility-isomorphisms plus some commutativity and associativity-restraints \linebreak (cf. \cite{SGA4}, XVIII, Définition 1.4.2 for the definition of a (strictly commutative) Picard category, or \cite{MB}, 14.4, axiome du pentagone et de l'hexagone):
There is an associativity-isomorphism $$a_{x,y,z}: (x \oplus y) \oplus z \to x \oplus (y \oplus z)$$
such that
$$\xy
(0,20)*{(w \oplus x) \oplus(y \oplus z)}="1";
(40,0)*{w \oplus(x \oplus(y \oplus z))}="2";
(25,-20)*{ \quad w \oplus((x \oplus y) \oplus z)}="3";
(-25,-20)*{(w \oplus(x \oplus y)) \oplus z}="4";
(-40,0)*{((w \oplus x) \oplus y) \oplus z}="5";
{\ar^{a_{w,x,y \oplus z}} "1";"2"}
{\ar_{1_w \oplus a _{x,y,z}} "3";"2"}
{\ar^{a _{w,x \oplus y,z}} "4";"3"}
{\ar_{a _{w,x,y} \oplus 1_z} "5";"4"}
{\ar^{a _{w \oplus x,y,z}} "5";"1"}
\endxy$$
commutes. There is a commutativity-isomorphism $c_{x,y}: x \oplus y \to y \oplus x$ such that
$$\xy
(-25,20)*{(x \oplus y) \oplus z}="0" ;
(9,20)="01" ;
(25,20)*{x \oplus(y \oplus z)}="1" ;
(40,0)*{x \oplus(z \oplus y))}="2" ;
(25,-20)*{(x \oplus z) \oplus y}="3" ;
(13,-20)="34" ;
(-25,-20)*{(z \oplus x) \oplus y}="4" ;
(-40,0)*{z \oplus (x \oplus y)}="5" ;
{\ar^{1_x \oplus c_{y,z}} "1";"2"}
{\ar_{a_{x,z,y}} "3";"2"}
{\ar^{c_{z,x} \oplus 1_y} "4";"34"}
{\ar_{a_{z,x,y} } "4";"5"}
{\ar^{a_{w \oplus x,y,z}} "5";"0"}
{\ar^{a_{x,y,z}} "0";"01"}
\endxy$$
commutes. It follows that a category such as this has a zero-object, has inverses etc. (see \cite{SGA4}, XVIII, 1.4.4).
In other words, a Picard category is a symmetric monoidal groupoid whose functor $- \oplus Q$ is an equivalence of categories for each object $Q$ in $\mc C$. It is moreover said to be strictly commutative if $c_{x,x}: x \oplus x \to x \oplus x$ is the identity. In general we denote by $\epsilon(x) = c_{x,x}$. An additive functor between Picard categories is defined to be a monoidal functor between Picard categories. \\
Observe we merely have isomorphisms $B \oplus (-B) \to 0$, not equality. For any exact category $\mathcal C$, the universal Picard category $V(\mathcal C)$ is a Picard category $\mc C$ with a functor $[]:(\mathcal C, is) \to V(\mathcal C)$  which is universal with respect to morphisms $T:(\mathcal C, is) \to P$ into Picard categories $P$, satisfying the following compatibility conditions:
\begin{description}
  \item[(a)] For any short exact sequence $$\mathcal A: 0 \to A' \to A \to A''
\to 0,$$ there is an isomorphism, functorial with respect to
isomorphisms of exact sequences, $$T(\mathcal A): T(A) \to T(A') \oplus
T(A'').$$
  \item[(b)] For any 0-object of $\mathcal C$, there is an isomorphism $T(0) \simeq 0$. \\
  \item[(c)] If $\phi: A \to B$ is an isomorphism, with exact sequence $0 \to A \to B \to 0$, the induced map $T(\phi)$ is the composite
$$T(A) \to T(0) \oplus T(B) \to T(B).$$
  \item[(d)] The functor $T$ is compatible with filtrations, \ie for an admissible filtration $A \subset B \subset C$, the diagram
$$\xymatrix{T(C) \ar[r] \ar[d] & T(A) + T(C/A) \ar[d] \\
T(B) + T(C/B) \ar[r] & T(A) + T(B/A) + T(C/B) }$$
is commutative (here the quotients are only defined up to unique isomorphism, but are well-defined by the other conditions).
\end{description}

In $\cite{determinant}$ it is shown that the functor $(\mathcal C, is) \to V(\mathcal C)$ factors as
\begin{equation}\label{functorsbetweencategories} (\mathcal C, is) \to (D^b(\mathcal C), q.i.) \to V(\mathcal C).
\end{equation}
Here $D^b(\mathcal C)$ is the derived category of $\mathcal C$ (supposed to be the full subcategory of a fixed abelian category), formed out of all complexes with bounded cohomology, where homotopic complexes are identified, and then localized at the thick subcategory of acyclic complexes. The extra suffix is to denote we consider the category where the objects are the same, but the morphisms are the quasi-isomorphisms, \ie morphisms in the category of complexes that induces an isomorphism in the derived category. \\

\subsection{Additional descriptions} \label{section:Determinantsontriangualated}

In this section we give some additional descriptions of the virtual category associated to the category of vector bundles over a fixed scheme $X$. But first, given a small exact category $\mc C$, there are two major spaces (or more properly, $S^1$-spectra in the sense of topology) that can be  associated, namely that of the Waldhausen and Quillen $K$-theory. However, they are naturally homotopic (see \cite{Thomason2} Theorem 1.11.2) hence have naturally equivalent fundamental groupoids and hence naturally equivalent virtual categories. If $X$ is a scheme, the Quillen $K$-theory space, the naive one, is the Waldhausen $K$-theory of bounded complexes of vector bundles in the category of coherent $\calo_X$-modules. The Thomason $K$-theory is the corresponding theory obtained by taking the Waldhausen $K$-theory of perfect complexes (see \cite{SGA6}, I.5.1. or \cite{Thomason2}, 2.1.1). $K^Q$ provides the "correct" definition of higher $K$-theory, however, $K^T$ provides the better definition in terms of functorial properties. We denote them by $K^Q$ and $K^T$ respectively. Also, denote by $\mathbf{P}(X)$ the category of (coherent) locally free sheaves on $X$. \begin{Prop}[\cite{Thomason2}, Corollary 3.9, Proposition 3.10] \label{prop:virtualvectorcomplexes} We have natural maps $$\Omega BQ(\mathbf{P}(X)) \to K^Q(X) \to K^T(X).$$ For general $X$, the first map is an homotopy-equivalence. Whenever $X$ has an ample family of line bundles (\cite{SGA6}, II. 2.3) or the resolution property (\ie any coherent sheaf is the quotient of a coherent locally free sheaf) the last map is also an homotopy equivalence. In particular, whenever $X$ has an ample family of line bundles we can define the virtual category as the fundamental groupoid of $K^T$.
\end{Prop}
\begin{remark} If $X$ has in particular an ample family of line bundles whenever $X$ is separated and regular, or is quasi-projective over an affine scheme. Lemma 3.5 of \cite{Thomason2} gives an additional list of spaces homotopy-equivalent to $K^T(X)$ for quasi-compact schemes $X$. \end{remark}
\begin{D} The fundamental groupoid of $K^Q$ and $K^T$ are denoted by $V(X)$ or $V^Q(X)$ and  $V^T(X)$ respectively. \end{D}
Both these definitions make $V^?$ into a contravariant functor from the category of schemes to the category of groupoids, $$V^?: Schemes \to \op{Grp},$$ via the pullback operation, $Lf^* = f^*$, and we have a natural transformation of functors $V^Q \to V^T$. It is not, in general, covariant with respect to even proper morphisms $f: X \to Y$ (cf. Proposition \ref{Prop:K-theory-covariance}). However, we have:
\begin{Prop}[\cite{Thomason2}, 3.16.4-3.16.6] \label{Prop:K-theory-covariance} The functor $V^T$ is a
covariant functor from the category
\begin{itemize}
  \item of Noetherian schemes and perfect (see \cite{Thomason2}, 2.5.2) proper
  morphisms. For example, any \lci (as defined in \cite{SGA6}, VIII, Proposition 1.7) proper morphism between Noetherian schemes.
  \item of quasi-compact schemes and perfect projective morphisms.
  \item of quasi-compact schemes and flat proper morphisms.
\end{itemize}
\end{Prop}
In view of the above we want to give a unified approach to determinant functors on Waldhausen categories and their associated virtual categories. sOriginally defined for exact categories (\cite{determinant},  4.3) we give the following definition:

\begin{D} \label{defn:determinantfunctor} Let $\textbf{A}$ be a Waldhausen category (\cite{Thomason2}, Example
1.3.6) with weak equivalences $w$. By $(\textbf{A}, w)$ we denote
the category having the same objects as $\textbf{A}$ but the
morphisms being weak equivalences. Given a Picard category $P$, a
determinant functor from $\textbf{A}$ to $P$ is a functor
$$[-]:(\textbf{A}, w) \to P$$ which satisfies the following constraints:
\begin{description}
  \item[(a)] For any cofibration exact sequence $$\Sigma: A' \rightarrowtail A \twoheadrightarrow A''$$ an isomorphism $\{\Sigma\}: [A] \to [A'] \oplus [A'']$ functorial with respect to weak equivalences of cofibration sequences.
  \item[(b)] For any object $A$, the cofibration sequence $\Sigma: A = A \twoheadrightarrow 0$ decomposes the identity-map as: $$ [A] \stackrel{\{\Sigma\}}\to [A] \oplus [0] \stackrel{\delta^R}\to [A]$$ where $\delta^R: [A] \oplus [0] \to [A]$ is given by the structure of $[0]$ as a right unit (see Lemma \ref{lemma:structureofunit} for a unicity and existence statement).
  \item[(c)] Suppose we have a commutative diagram $$\xymatrix{\Sigma': & A' \ar@{>->}[r] \ar@{>->}[d] & B' \ar@{->>}[r] \ar@{>->}[d] & C' \ar@{>->}[d]  \\ \Sigma: & A \ar@{>->}[r] \ar@{->>}[d] & B \ar@{->>}[r] \ar@{->>}[d] & C \ar@{->>}[d] \\ \Sigma'': & A'' \ar@{>->}[r] & B'' \ar@{->>}[r] & C''\\ & \Sigma_A & \Sigma_B & \Sigma_C}$$ were all the vertical and horizontal lines are cofibration sequences. Then the   diagram $$\xymatrix{ [B'] \oplus [B''] \ar[d]^{\{\Sigma'\} \oplus \{\Sigma''\}} & \ar[l]^{\{\Sigma_B\}}[B] \ar[r]^{\{\Sigma\}} & [A] \oplus [C] \ar[d]^{\{\Sigma_A\} \oplus \{\Sigma_C\}} \\ [A'] \oplus [C'] \oplus [A''] \oplus [C''] \ar[rr] & & [A'] \oplus [A''] \oplus [C'] \oplus
[C'']}$$ is commutative. \\ It is furthermore said to be commutative if the following holds: \\
  \item[(d)] The triangle $$\xymatrix{[A'] \oplus [A''] \ar[rr] \ar[rd] & & \ar[ld] [A''] \oplus [A'] \\ & [A' \bigcup
  A'']}$$ commutes.
\end{description} \end{D}
We record the following lemma:
\begin{TL} \label{lemma:structureofunit} Suppose $[]: (\textbf{A}, w) \to P$ is
a determinant functor and $P$. Then for any 0-object of
$\textbf{A}$, $[0]$ has the structure of a unit in $P$, \ie there
are canonical isomorphisms $\delta^L: [0] \oplus B \simeq B$ and
$\delta^R: B \oplus [0] \simeq B$. In particular, there is a canonical
isomorphism $[0] \simeq 0$ with any unit object 0 of $P$.
\end{TL}
\begin{proof} Applying $[]$ to the cofibration sequence
$$0 \rightarrowtail 0 \twoheadrightarrow 0$$
we obtain an isomorphism $[]: [0] \oplus [0] \simeq [0]$. By
\cite{Saavedra}, 2.2.5.1 $[0]$ has a unique structure of a unit
such that $[0] = \delta^R([0]) = \delta^L([0])$.
\end{proof}

We note the following theorem which extends Deligne's categorical
description of the virtual category:

\begin{T} \label{thm: waldhausen-determinant} Let \textbf{A} be a small Waldhausen category
with weak equivalences $w$. Then there is a universal category for
determinant functors:

$$[-]:(\textbf{A}, w) \to V(\textbf{A}).$$ More precisely, for any Picard
category $P$, the category of determinant functors is equivalent to
the category of additive functors $V(\textbf{A})\to P$. Moreover,
this category is the fundamental groupoid of the Waldhausen
$K$-theory space of $\textbf{A}$.

\end{T}
\begin{proof} The proof is essentially by definition. Recall that
the Waldhausen $K$-theory space is the loop space of the geometric
realization of the bisimplicial set $N_\bullet wS_\bullet
\mathbf{A}$ where $wS_p \mathbf{A}$ is the category whose objects
are, for $0 \leq i \leq j \leq p$, sequences $A_i \rightarrowtail
A_j$ of cofibrations with $A_0 = 0$ and with choices of quotients
$A_j/A_i$, and natural compatibility with composition so that $A_i \rightarrowtail A_j \rightarrowtail A_k$ coincides with $A_i \rightarrowtail A_k$ for $i \leq j \leq k$, and whose morphisms between two objects $A$ and $A'$ are
given by weak equivalences $A_i \to A_i'$ making all the diagrams
commute. $N_p wS_q \mathbf{A}$ is the $p$-nerve of the category
$wS_q \mathbf{A}$. The categories $wS_0\mathbf{A}, wS_1\mathbf{A},
wS_2\mathbf{A}$ are, respectively, the trivial category, the
category of objects of $\textbf{A}$ and weak equivalences as
morphisms, and the category of cofibration sequences with weak
equivalences of cofibration sequences as morphisms.

The geometric realization in question is the (left-right) realization
$$|q \mapsto |p \mapsto N_p wS_q \mathbf{A}||.$$ Thus we obtain from
the above description that the "0-simplices" are simply reduced to a
point  and the "1-simplices" in the $S_\bullet$-direction is
obtained by adjoining the set
$$|p \mapsto N_p wS_1 \mathbf{A}| \times \Delta^1.$$
This defines a canonical map $|wS_1| \wedge S^1 \to |N_\bullet
wS_\bullet \textbf{A}|$, and by adjunction a map $|wS_1| \to \Omega
|N_\bullet wS_\bullet \textbf{A}| = K(\textbf{A})$. By applying the
fundamental groupoid-functor we obtain a functor $[]: \textbf{A} \to
(w^{-1} \textbf{A}, w) = \pi_f(|wS_1 \textbf{A}|) \to \pi_f
(K(\textbf{A}))$, by sending an object to the loop represented by $A
\in N_0 wS_1 \mathbf{\textbf{A}}$.
We verify that this is a determinant functor: \\
Axiom a: A cofibration sequence $$\Sigma: A \rightarrowtail B
\twoheadrightarrow C$$ defines an element in $N_0 wS_2(\textbf{A})$,
and the face-maps to $N_0 wS_1(\textbf{A})$ are given by $\partial_0
\Sigma = A, \partial_1 \Sigma = B, \partial_2 \Sigma = C$, thus
providing a path from $[B]$ to $[A] + [C]$. A weak equivalence of
cofibration sequences defines an element in
$$N_1 wS_2 \textbf{A}$$ whose faces are in $N_1 wS_1 \textbf{A}$, which provides the necessary path.\\
Axiom b: This is just a simplicial identity corresponding to the degeneracy $N_0 wS_1(\textbf{A}) \to N_0 wS_2(\textbf{A}) \to N_0 wS_1(\textbf{A}), A \mapsto [A \to A \to 0] \mapsto A$.\\

Axiom c: We first show that the commutativity can be rephrased as:
if $A \rightarrowtail B \rightarrowtail C$ of cofibrations then
  $$\xymatrix{[C] \ar[r] \ar[d] & [A] + [C/A] \ar[d] \\
              [B] + [C/B] \ar[r] & [A] + [B/A] + [C/A]}$$
  commutes. This is clear since
  $$\xymatrix{& B \ar@{>->}[rd] & \\
  A \ar@{>->}[ur] \ar@{>->}[rr]& & C}$$
  is a 3-simplex (an object in $N_0 wS_3 \textbf{A}$) and provides the necessary relationship between morphisms induced from
  the 2-simplices in $N_0 wS_2 \textbf{A}$ (one also needs to use the commutativity in Axiom d, which is easy). For the full theorem we use that the two 3-simplices $A' \rightarrowtail B' \rightarrowtail B$
  and $A' \rightarrowtail A \rightarrowtail B$ are glued together along the 2-simplex $A' \rightarrowtail B$. \\
We need to prove this construction is universal. Let $[]:(\mc A, w)
\to P$ be a determinant functor. We construct maps $f_n:wS_n
(\textbf{A}) \to \prod_{i = 1}^n P$, where the latter denotes
the naive sum of Picard categories with indice-wise objects, $\hom$s
and additions. We put $$f_n(0 \rightarrowtail A_1 \rightarrowtail
\ldots \rightarrowtail  A_n) = ([A_1], [A_2]-[A_1], \ldots, [A_n] -
[A_{n-1}]).$$ Equip $\prod^\bullet P$ with the structure of a
simplicial category by the "bar simplicial resolution"-structure; the "$n$-simplices" are given by the product of categories $\prod_1^n P$ and
that $d_0(g_1, \ldots, g_n) = (g_2, \ldots, g_n)$ and for $0 < i <
n$, $$d_i(g_1, \ldots, g_n) = (g_1, \ldots, g_{i-1}, g_i g_{i+1},
g_{i+2}, \ldots, g_n)$$ and $$d_n(g_1, \ldots, g_{n-1}, g_n) = d_n(g_1, \ldots, g_{n-1}).$$ The
face-maps are given by $$s_i(g_1, \ldots, g_n) = (g_1, \ldots, g_i, 0, g_{i+1}, \ldots, g_n).$$ By taking nerves, we obtain a map of
bisimplicial sets $N_\bullet wS_\bullet (\mathbf A) \to N_\bullet
\prod^\bullet P$.
It is readily viewed that the simplicial category $\prod^\bullet P$ has a natural augmentation to the "constant" simplicial category $P$ and as such a natural morphism of nerves $N_\bullet \prod^\bullet P \to N_\bullet P$. Upon applying geometric realizations and fundamental groupoids we obtain a canonical map $K(\mc A) \to |P|$ of topological spaces. Or rather a map $d(K(\mc A)) \to S|P|$ of simplicial sets, where $d$ denotes the diagonal of a bisimplicial set and $S$ is the functor associating to a simplicial set its singular complex. Applying the functor $\pi_f$ to this gives us the required canonical functor, in view of the fact that in general $\pi_f{S|P|} = \pi_f P = P$.
\end{proof}

\begin{remark} \label{remark:Picardcontinuousmaps} Let $\textbf{A}$ be any (small) saturated Waldhausen
category (see \cite{Thomason2}, Definition 1.2.5), so that in
particular the localization $w^{-1} \textbf{A}$ is well behaved. Then
any functor from the groupoid $F: (w^{-1} \textbf{A}, w) \to P$ to a
(small) groupoid $P$ corresponds in fact to a map of topological
spaces $|F|: |wS_1 \textbf{A}| \to |P|$, and the functor $F$ is
recovered by applying the functor $\pi_f$. This follows from the fact that
for two simplicial sets $X$ and $Y$, there is a natural bijection
$\hom_{Top}(|X|, |Y|) = \hom_{S}(X, S|Y|)$ and that the fundamental
groupoid of a simplicial set $Y$ and $S|Y|$ are in fact the same;
see \cite{SimplicialHT}, chapter I, Proposition 2.2 and section 8.
It is also clear that the fundamental groupoid of the nerve of a
groupoid is in fact the same groupoid. Also, any functor
$F:\pi_f(K(\textbf{A})) \to P$ of (small) Picard categories
correspond to a map of topological spaces $K(\textbf{A}) \to |P|$
by, for example, sending an $n$-simplex of the simplicial set $|wS_q
A\textbf{A}|$ of the form  $A_1 \rightarrowtail A_2 \rightarrowtail
\ldots \rightarrowtail A_{n+1}$ to the $n$-th nerve $F(A_{n+1})
\simeq F(A_{n+1}/A_n) \oplus F(A_{n}) \simeq \ldots \simeq
F(A_{n+1}/A_{n}) \oplus F(A_{n}/A_{n-1}) \oplus \ldots \oplus F(A_1)$ of $P$. The
functor $|F|$ can again be recovered by applying the functor
$\pi_f$. Thus, the above problem of describing determinant functors
can likewise be formulated as a lifting-problem of certain maps of
topological spaces. We ask for which continuous functions $f:|wS_1
\textbf{A}| \to |P|$ there is a lift as in the diagram below
$$\xymatrix{|wS_1 \textbf{A}| \ar[r] \ar[dr] & K(\textbf{A}) \ar@{-->}[d]^{\exists} \\
& |P|.}$$ We leave the precise reformulation to the interested
reader, as it will not be used in the rest of this article, except for giving the language for formulating the following definition.
\end{remark}
\begin{D} \label{defn:determinantfunctorimage} Given a determinant functor $F: P \to P'$ of Picard categories, we can define the image in the obvious way as the essential image of the functor and similarly the kernel of $F$. We denote them by $\op{Im} F$ and $\op{ker} F$ respectively. If $F: P \to P'$ is a determinant functor which is faithful as a functor (for example, the inclusion of the essential image of a determinant functor in the target category), we can define the cokernel of $F$, denoted $P'/P$ as the fundamental groupoid of the mapping cone construction in topology via the geometric realization as in \ref{remark:Picardcontinuousmaps}. As such it sits in a long exact sequence
$$0 \to \pi_1(P) \to \pi_1(P') \to \pi_1(P/P') \to \pi_0(P) \to \pi_0(P') \to \pi_0(P'/P) \to 0$$
of abelian groups. Alternatively, it is the category whose objects are the same as that of $P'$, and $$\xymatrix{\Hom_{P'/P}(A,A') & = & \{B, B' \in \op{ob} P', f \in \Hom_{P'}(A + F(B), A'+F(B')) \\ & & \hbox{s.th. }f \sim f', \hbox{ if } \exists C, C' \in \op{ob}(P'), f - f' \in F\Hom_P(C, C')\}.}$$ In general the cokernel of $F$ is the cokernel of the inclusion of the image in the target category. All these categories have natural structures of Picard categories and the induced functors are determinant functors, and are moreover functorial with respect to natural transformations.
\end{D}
\begin{remark} Since any exact category can be equipped with the structure of a biWaldhausen category, where the weak equivalences are the isomorphisms and the cofibrations are the admissible monomorphisms, it is clear that the above definition generalizes that of Deligne \cite{determinant}, 4.3. It is a simple exercise to verify that in this case the above axioms for determinant functors are equivalent to those given in loc.cit. \end{remark}
We also have the following stronger assertion:
\begin{Prop} [\cite{Waldhausen} Theorem 1.9] The Waldhausen $K$-theory spectrum of an exact category $\mc E$, $K(\mc E)$, is naturally homotopy-equivalent to the $K$-theory spectrum of Quillen. A fortiori it induces an equivalence of fundamental groupoids and Picard categories.
\end{Prop}

\label{complicialbiWaldhausen}\begin{Prop} [\cite{Thomason2}, 1.9.6]
Suppose in addition that $\textbf{A}$ is complicial biWaldhausen so that it is a full subcategory of $C(\mc A)$ for an abel an category $\mc A$. Furthermore suppose that it is closed under taking exact sequences in $C(\mc A)$, is closed under finite degree shifts and $co(\textbf{A})$. Then  $Ho(\textbf{A}) = w^{-1} \textbf{A}$ is a triangulated category and admits a calculus of fractions.
\end{Prop}

\begin{remark} By \cite{Thomason2}, Theorem 1.9.2, we can suppose that cofibrations are the degree-wise admissible monomorphisms with quotients in $\textbf{A}$. \end{remark}

\begin{Prop} \label{prop:determinanttriangulated} With the above notation, a determinant functor $[]: (\textbf{A}, w) \to P$ admits the following equivalent description as a functor factoring via $D(\textbf{A}) = Ho(\textbf{A}):= w^{-1} \textbf{A}$: \\
(a'): For any distinguished triangle $\Sigma: A \to B \to C \to A[1]$ there is an isomorphism functorial with respect to
isomorphisms (in $D(\textbf{A})$):
$$\{\Sigma\}: [B] \simeq [A] \oplus [C].$$
(b'): For any object $A$, the distinguished triangle $\Sigma: A = A \to 0 \to A[1]$ decomposes the identity-map as:
$$\xymatrix{[A] \ar[r]^{\{\Sigma\}} & [A] \oplus [0] \ar[r]^{\delta^R} & [A]}.$$
(c'): For any distinguished triangle of distinguished triangles,
\ie a diagram of the form:
$$\xymatrix{\Sigma_A: & A' \ar[r] \ar[d] & A \ar[r] \ar[d] & A'' \ar[r] \ar[d] & A'[1] \ar[d] \\
\Sigma_B: & B' \ar[r] \ar[d] & B \ar[r] \ar[d] & B'' \ar[r] \ar[d] & A'[1] \ar[d]\\
\Sigma_C: & C' \ar[r] \ar[d] & C \ar[r] \ar[d] & C'' \ar[r] \ar[d] & A'[1] \ar[d]\\
\Sigma_A[1]: & A'[1] \ar[r] & A[1] \ar[r] & A''[1] \ar[r] & A'[2] \\
& \Sigma' & \Sigma & \Sigma'' & \Sigma'[1]}$$ where all the rows and
columns are distinguished triangles, the following diagram is
commutative:
  $$\xymatrix{ [B'] \oplus [B''] \ar[d]^{\{\Sigma'\} \oplus \{\Sigma''\}} & \ar[l]^{\{\Sigma_B\}}[B] \ar[r]^{\{\Sigma\}} & [A] \oplus [C] \ar[d]^{\{\Sigma_A\} \oplus \{\Sigma_C\}} \\
[A'] \oplus [C'] \oplus [A''] \oplus [C''] \ar[rr] & & [A'] \oplus [A''] \oplus [C'] \oplus
[C'']}.$$\\ We moreover say that the determinant functor is commutative if
(d') The natural triangles $$A \to A \oplus B \to A \to A[1]$$ and
$$B\to B \oplus B \to A \to B[1]$$ induce a commutative diagram
$$\xymatrix{[A \oplus B] \ar[r] \ar[d] & \ar[d] [B \oplus A] \ar[d] \\
[A] \oplus [B] \ar[r] & [B] \oplus [A].}$$
\end{Prop}
\begin{proof} We can suppose by the above remark that the cofibrations are given by degree-wise split monomorphisms, and these yield all distinguished triangles in $D(\textbf{A})$. It is thus clear that the above data $(a') - (c')$ determine the data $(a) - (c)$, and that the data $(d)$ and $(d')$ are equivalent, so we show the converse statement. Since $D(\textbf{A})$ admits a calculus of fractions it is immediate to verify that if a cofibration sequence $\Sigma: A' \rightarrowtail B' \twoheadrightarrow C'$ is isomorphic to a distinguished triangle $A \to B \to C \to A[1]$ we have an induced isomorphism $[B'] \simeq [A'] \oplus [C']$ such that the obvious diagram commutes. By $(a)$ it does not depend on the choice of $\Sigma$. $(b')$ is clearly equivalent to $(b)$. To establish that the data of $(c)$ determine that of $(c')$, we first notice that if $u: A \to B$ is any morphism in $D(\textbf{A})$, and if we have two choices of cones of $u$, $C$ and $C'$, with an isomorphism $f: C\to C'$, by $(a')$ applied to the diagram
$$\xymatrix{A \ar[r] \ar@{=}[d] & B \ar[r] \ar@{=}[d] & C \ar[r] \ar[d]^f & A[1] \ar@{=}[d] \\ A \ar[r] & B \ar[r] & C' \ar[r] & A[1] }$$ there is an isomorphism $[C] = [C']$ which does not depend on the choice of $f$. Hence the object $[cone(u)]$ is determined up to unique isomorphism, as
opposed to $cone(u)$, and for any distinguished triangle $A \stackrel{u}\to B \to C \to A[1]$ a canonical isomorphism $[C] = [cone(u)]$. Furthermore, any diagram $$\xymatrix{A' \ar[r] \ar[d] & A \ar[d] \\ B' \ar[r] & B}$$ can be completed into a diagram of the form in $(c')$ by taking mapping cones and it follows that we can assume that our diagram is of that form. Lastly, if we have a map $u: A \to B$, the map $B \to cone(u)$ can be chosen to be represented by a cofibration and we are reduced to the case of cofibrations.
\end{proof}

Determinants on (small) triangulated categories were also studied in \cite{detertriangulated},  \cite{Knudsen} and \cite{K1Waldhausen}, where similar results were obtained. In particular we have:
\begin{Cor}[Knudsen, \cite{Knudsen}] Let $i: \mc E\to \mc A$ be an exact fully faithful embedding of an exact category $\mc E$ in an abelian category $\mc A$, such that for any morphism in $\mc E$ which is an epimorphism in $\mc A$, is admissible in $\mc E$. Denote by $C(\mc E)$ the full subcategory of bounded complexes of the category of complexes in $\mc A$. Then we have a natural equivalence of categories between the virtual category of $(\mc E, is)$ and the virtual category of $(C(\mc E), q.i.) $ of complexes in $E$ with quasi-isomorphisms in $\mc A$.
\end{Cor}
\begin{proof} Equip the category $C(\mc E)$ with the structure of
a complicial biWaldhausen category where the weak equivalences are
given by quasi-isomorphisms and the cofibrations are either of the
two following: degree-wise admissible monomorphisms or degree-wise split monomorphisms whose quotients lie in $C(\mc E)$. Denote the
corresponding biWaldhausen categories by $\textbf{E}$ and
$\widetilde{\textbf{E}}$. By \cite{Thomason2}, Theorem 1.11.7, we have
natural homotopy-equivalences
$$K(\mc E) \simeq K(\textbf{E}) \simeq K(\widetilde{\textbf{E}})$$
and hence equivalent virtual categories. Moreover, this does not
depend on the choice of $\mc A$.
\end{proof}
\mps{Notice, however, that the proof of
the theorem in \loccit resembles that of \cite{Knudsen}, in
particular in the usage of the "passage to the length two
complexes"-induction step.}
If $i: \mc E \to \mc A$ is the fully faithful Gabriel-Quillen
embedding reflecting exactness (see \cite{Thomason2}, Appendix A),
or if $\mc E$ is the category of coherent vector bundles and $\mc A$
is the category of coherent sheaves respectively on a scheme, $i$
satisfies the above hypothesis.

\begin{Cor} \cite{K1Waldhausen} Let $\textbf{A}$ be a small complicial biWaldhausen category. Then $K_0$ and $K_1$ are functorially (with respect to triangulated functors, \ie functors preserving the above structures) determined by the structure of a triangulated category of the homotopy category $Ho(\textbf{A}) = D^b(\textbf{A})$ and its isomorphisms. In particular one obtains a description of the $K_1$ of a Waldhausen category in terms of the associated derived category (see \loccit for a precise statement).
\end{Cor}

\section{Some fundamental operations on virtual categories} \label{chapter:algebraicstacks}
\subsection{Various categories and some fundamental properties}

We will freely use the language of Appendix \ref{appendix:algebraicstacks} in this chapter where we expand slightly on the concept of a virtual category of an algebraic stack. We will always consider the a stack as a simplicial sheaf via the extended Yoneda functor \ref{defn:extyoneda}. Also, for the purposes of this section, all algebraic stacks are separated locally of finite type over some (non-fixed) Noetherian scheme $S$. \\
\begin{D} Given an algebraic stack $\mc X$, there are for our purposes four main candidates for virtual categories one might consider, namely any one of the following Picard categories
\begin{enumerate}
  \item the virtual category of locally free sheaves on $\mc X$, $V(\mc X) = V_{naive}(\mc X)$.
  \item the virtual category of coherent $\calo_\mc X$-modules on $\mc X$, $C(\mc X).$
  \item if $\mc Z$ is a closed substack of $\mc X$, the fundamental groupoid of the $K$-theory of the category of finite complexes of vector bundles on $\mc X$ with support on $\mc Z$, $V^\mc Z(\mc X)$.
  \item if $\mc Z$ is a closed substack of $\mc X$, the fundamental groupoid of the $K$-theory of the category of complexes of vector bundles on $\mc X$ with support on $\mc Z$, $C_\mc Z (\mc X)$.
  \item the fundamental groupoid of $K^{sm}(\mc X)$, the cohomological virtual category, $W(\mc X)$.
  \item the fundamental groupoid of $G^{sm}(\mc X)$, the coherent cohomological virtual category, $WC(\mc X)$.

\end{enumerate}
\end{D}
By the remarks concluding the Appendix \ref{appendix:algebraicstacks} we have additive functors of fibered Picard categories, $V(-) \to W(-)$ and $C(-) \to WC(-)$. Notice that since the automorphism-group of any object of $W(\mc X)$ or $WC(\mc X)$ is a $\bb Q$-vector space they are automatically strictly commutative.
\begin{D} Since $K^{sm}$ is flabby, to give operations involving $W(\mc X)$ it is sufficient to construct functorial homotopies on the $K$-theory
spaces of the vertices of simplicial algebraic space $\mc N(X/\mc X)$ for some presentation of $\mc X$. The same remark applies to $WC(\mc X)$. We will say that any such constructed operations are given by \textit{cohomological descent}.
\end{D}
Given a morphism $F: \mc X \to \mc Y$ of algebraic stacks locally of finite type over a Noetherian scheme $S$, recall that for a coherent sheaf
$\mc F$ we can define $R^i F_* \mc F$ by a Cech-cohomology argument (compare \cite{HodgeIII}, Définition 5.2.2.). \mps{Is this not clear? I'm just saying one defines derived functors via Cech cohomology.} We know by \cite{Olsson-propercovering}, Theorem 1.2, that whenever $F$ is moreover proper,
$R^iF_* \mc F$ is coherent whenever $\mc F$ is coherent. Suppose in addition that $F$ is of finite cohomological dimension so that $R^i F_*(\mc F) = 0$ for large enough $i$. Then the usual formula
$$RF_*(\mc F) = \sum (-1)^i R^i F_* \mc F$$
defines a pushforward on $RF_*: C(\mc X) \to C(\mc Y)$. It is more subtle to define the corresponding functor $WC(\mc X) \to WC(\mc Y)$. If $F: \mc X \to \mc Y$ is a proper morphism, and given a proper surjective morphism $X \to \mc X$ with $X$ a scheme, we obtain a diagram of
$$\xymatrix{\mc N(X/\mc X) \ar[dr]^q \ar[d]^p \\
\mc X \ar[r]^F & \mc Y }$$
with proper morphisms and applying the functor $G^{sm}()$ we obtain a diagram
$$\xymatrix{G^{sm}(\mc N(X/\mc X)) = G(\mc N(X/\mc X)) \ar[dr]^{q_*} \ar[d]^{p_*} \\
G^{sm}(\mc X)  & G^{sm}(\mc Y) }.$$
By \cite{Toen}, Théorème 2.9, given a proper surjective morphism $X \to \mc X$ with $X$ a scheme and $\mc X$ is Deligne-Mumford, there is a weak equivalence $G(\mc N(X/\mc X)) \to G^{sm}(\mc X)$. Applying the fundamental groupoid-construction thus gives an equivalence of categories
$\pi_f(G^{sm}(\mc N(X/\mc X))) \to WC(\mc X)$ and we define $RF_* = q_*(p_*)^{-1}: WC(\mc X) \to WC(\mc Y)$ (compare \cite{Toen-Riemann-Roch}, Section 3.2.2). We have essentially proved:
\begin{Prop} Suppose $F: \mc X \to \mc Y$ is a proper of finite cohomological dimension morphism of separated Deligne-Mumford stacks of finite type over a Noetherian base-scheme $S$. It is possible to define a functor $RF_*: WC(\mc X) \to WC(\mc Y)$ such that the diagram
$$\xymatrix{C(\mc X)  \ar[r] \ar[d]^{RF_*} & WC(\mc X) \ar[d]^{RF_*} \\
C(\mc Y)  \ar[r]  & WC(\mc Y) }$$
is commutative up to canonical equivalence of functors. \\
\end{Prop}
\begin{proof} The statement is clear as soon as we can show that there is always a choice of a proper surjective $X \to \mc X$ with $X$ a scheme. It is clearly independent of such a choice. But this is \cite{Olsson-propercovering}, Theorem 1.1, which moreover shows we can pick $X$ to be quasi-projective over $S$.
\end{proof}

The following uses a standard argument factorizing a projective morphism as a closed immersion and a projective bundle projection, we refer to \cite{Riemann-RochAlgebra}, chapter V for the definition.

\begin{Prop} \label{Prop:K-pushforward} Suppose $F: \mc X \to \mc Y$ is a (representable) projective \lci morphism of algebraic stacks with $\mc Y$ quasi-compact and $\mc Y$ has the resolution property, \ie any coherent sheaf is the quotient of a locally free sheaf. Then there is a natural  functor $$RF_*: V(\mc X) \to V(\mc Y)$$ compatible with the functor defined on $C$ under the additive functor $V(-) \to C(-)$.
\end{Prop}

\begin{remark} Whenever we are working in a category of stacks where perfect complexes can be used to define algebraic $K$-theory the above is just a consequence of preservation of perfectness of a complex under proper \lci morphisms. The compatibility under composition is given by Grothendieck's spectral sequence.
\end{remark}
\mps{I ignore if the conditions of the proposition implies this.}
Similarly, if $E$ is a vector bundle on $\mc Y$, and $F: \mc X \to \mc Y$ is any morphism, we define a functor $LF^*: V(\mc Y) \to V(\mc X)$ via $LF^* [E] = [F^* E]$.

Let us just recall the usual definition of the basechange morphism, which always exists. Let $$\xymatrix{X' \ar[r]^{g'} \ar[d]^{f'} & Y' \ar[d]^f \\
X \ar[r]^g & Y }$$ be a Cartesian diagram of schemes. By adjointness, we have an equality of morphisms in the derived category of quasi-coherent complexes \mps{I hope this is ok. I'm always a bit scared of various derived categories not being equivalent} schemes; $$Hom(Lf^* Rg_* E, Rg'_* Lf'^* E) = Hom(Rg_* E, Rf_* Rg'_* Lf'^* E)$$ and since $Rf_* Rg'_* \simeq Rg_* Rf'_*$ this is equal to $$Hom(Rg_* E, Rg_* Rf'_* Lf'^* E).$$
By the adjunction morphism $E \to Rf'_* Lf'^* E$ we thus obtain a map $$Hom(Rg_* E,Rg_* E) \to Hom(Lf^* Rg_* E, Rg'_* Lf'^* E).$$
The basechange morphism is the morphism which is the image under the identity-map on the left-hand-side.
\begin{D} Let $$\xymatrix{X' \ar[r]^{g'} \ar[d]^{f'} & Y' \ar[d]^f \\
X \ar[r]^g & Y }$$ be a commutative diagram of schemes. We say the diagram is transversal or Tor-independent or that $X$ and $Y'$
are transversal or Tor-independent over $Y$ (\cite{SGA6}, III, Définition 1.5) if the diagram is a Cartesian diagram of schemes, with $Y$ quasi-compact, $f$ quasi-compact and quasi-separated and if for any $x \in X, y' \in
Y'$ mapping to the same point $y \in Y$, we have
$$Tor_i^{\calo_{Y,y}}(\calo_{X,x}, \calo_{Y', y'}) = 0, \hbox{ for } i > 0,$$ and
$f$ is of finite Tor-dimension.
\end{D}
\begin{TL}\label{lemme:base-changes}[SGA6, IV 3.1] Let
$$\xymatrix{X' \ar[r]^{g'} \ar[d]^{f'} & Y' \ar[d]^f \\
X \ar[r]^g & Y }$$ be a transversal diagram, and let $E \in D^b(X)$  be a
complex with quasi-coherent cohomology. In this case the basechange morphism is an
isomorphism
$$Lf^* Rg_* E \simeq Rg'_* Lf'^* E.$$
\end{TL}

Since it is natural it also satisfies descent with respect to any smooth equivalence relationship and
thus we have
\begin{Cor} \label{Cor:base-changestacks} Let $$\xymatrix{X' \ar[r]^{g'} \ar[d]^{f'} & Y' \ar[d]^f \\
X \ar[r]^g & Y }$$ be a transversal Cartesian diagram of quasi-compact algebraic stacks with the resolution property and representable morphisms,
$f$ and $f'$ \lci projective morphisms. Then there is a natural transformation $$Lg^* Rf_* = Rf'_* {Lg'}^*$$ of functors $V(\mc Y') \to V(\mc X)$.
\end{Cor}
\begin{proof} From the above one readily obtains that if a vector bundle $E$ is $f_*$-acyclic, $f_* E$ is also $g^*$-acyclic and that ${g'}^* E$ is $f'_*$-acyclic, inducing
an isomorphism $g^* f_* E \to f'_* {g'}^* E$. If $f$ is a projective bundle-projection we can, by Theorem \ref{thm:K-propertiesstack}, assume that $E$ is of the form
$\sum f^* E_i \otimes \calo(-i)$ which is a sum of $f_*$-acyclic objects. In the case $f$ is a closed immersion $f_*$ is automatically exact. The general case is obtained via the composition
of the two which by standard techniques is seen to be independent of the choice of the factorization.
\end{proof}
The following will be used later
\begin{TL} \label{lemma:basechange-compatability}
The following diagrams are commutative whenever all of the morphisms are defined:
\begin{enumerate}
             \item Let $$\xymatrix{X'' \ar[r]^{g''} \ar[d]^{e'} & Y'' \ar[d]^e \\
X' \ar[d]^{f'} \ar[r]^{g'} & Y' \ar[d]^f \\
X \ar[r]^g  & Y}$$

be the composition of two transversal cartesian diagrams. Then the third diagram is also transversal and the diagram
$$\xymatrix{Lg^* R(f e)_* \ar[rr] \ar@{=}[d] & & R(f' e')_* L{g''}^* \ar@{=}[d] \\
Lg^* Rf_* Re_* \ar[r] & Rf'_* L{g'}^* Re_* \ar[r] & Rf'_* Re'_* L{g''}^* }$$
is commutative.
             \item Let $$\xymatrix{X'' \ar[r]^{h'} \ar[d]^{f''} & X' \ar[r]^{g'} \ar[d]^{f'} & X \ar[d]^f \\
Y'' \ar[r]^h & Y' \ar[r]^g & Y}$$
be composition of two transversal cartesian diagrams.Then the third diagram is also transversal and
the diagram
$$\xymatrix{L(g h)^* Rf_* \ar[rr] \ar[d]  & & Rf'_* L(g' h')^* \ar[d] \\
Lh^* Lg^* Rf_* \ar[r] & Lh^* R{f'}_* L{g'}^* \ar[r]  & R{f'}_* L{h'}^* L{g'}^* }$$
\end{enumerate}
\end{TL}
\begin{proof}  Left to the reader (compare the unproved result of \cite{SGA4}, XII, Proposition 4.4).
\end{proof}

The following is trivial:
\begin{TL} [Projection formula] \label{lemma:projectionformula} Let $f: \mc X \to \mc Y$ be a \lci projective morphism of algebraic stacks with the resolution property. Suppose $F$ is a virtual bundle on $\mc Y$ and $E$ is a virtual bundle on $\mc X$. Then there is a functorial isomorphism $Rf_*(E \otimes Lf^* F) \to Rf_* (E) \otimes F$ compatible with transversal basechange, \ie for a diagram as in Corollary \ref{Cor:base-changestacks}, there is a commutative diagram
$$\xymatrix{Lg^* Rf_*(E \otimes Lf^* F) \ar[r] \ar[d] & Lg^* (Rf_* (E) \otimes Lf^* F) \ar[d] \\
Rf'_* (L{g'}^* E \otimes L{g f'}^* F) \ar[r] & R{f'}_*(L{g'}^* E) \otimes Lg^* Lf^* F)}$$
where the horizontal lines are given by the projection-formula and the vertical lines are given by  basechange. Moreover it is stable under composition in the naive way.
\end{TL}
\begin{remark} \label{remark:projectionformulacoherent} We also have a projection formula isomorphism in the case instead of the virtual category of vector bundles we consider the virtual category of coherent sheaves as input for $E$.
\end{remark}

\subsection{A splitting principle} \label{section:splittingprinciple}

Below we sketch a criterion for when we can descend a morphism on the level of the complete flag-variety to the base \footnote{recall that a flag is a sequence of sub-vector bundles $\mc E_0 \subset \mc E_1 \subset \ldots \subset \mc E_n$ whose successive quotients $\mc E_{i+1}/\mc E_i$ are also vector bundles. It is furthermore complete if each such quotient is a line bundle. }. First, let $E$ be an
vector-bundle of rank $e+1$ on a separated algebraic stack $\mc X$. Then $p^1: Y_1 = \bb P(E) \to
\mc X$ is a projective bundle which on which we have a
canonical sub-line bundle $\calo(-1)$, and a canonical quotient-bundle of ${p^1}^*
E$. Repeating this construction with the quotient-bundle, we
eventually obtain a map $p: \mc Y = \mc Y_e \to \mc Y_{e-1} \to \ldots \to \mc Y_1
\to \mc Y_0 = \mc X$, where the top space is the complete flag-variety of $E$
on $\mc X$, which also comes equipped with a canonical complete flag. Suppose $P$ is a contravariant functor from the category of
separated algebraic stacks to the category of Picard categories such that for any $\mc X$ there is a distributive additive functor
$V(-) \times P(-) \stackrel{\otimes}\to P(-)$
moreover satisfying the projective bundle axiom; for any $\mc X$, the functor
$$\times_{i=0}^{e} P(\mc X) \to P(\bb P(E))$$
given by $(f_i)_{i=0}^{e} \mapsto \sum_{i=0}^{e} {p^1}^* f_i \otimes \calo(-i)$ is an equivalence of categories.
Then the following is a  version of an observation of Franke in terms of Chow categories of ordinary schemes (see the article by J. Franke, "Chern Functors" in
\cite{AAG}, 1.13.2):

\begin{T} \label{thm:splittingprinciple} [Splitting principle] Let $p_1,p_2: \mc Y \times_\mc X \mc Y \to \mc Y$
be the two projections, and $r = p p_1 = p p_2$. Then \\
(a) $p^*: P(\mc X) \to P(\mc Y)$ is faithful. \\
(b) Suppose we have two objects $A, B \in \op{ob} P(X)$, and $f: p^*
A \to p^* B$ a morphism in $\Hom_{P(\mc Y)}(p^* A, p^* B)$, then $f$
comes from a (unique) morphism $h: A\to B$ if and only if $p_1^* (f)
= p_2^* (f)$ in $\op{Hom}_{P(\mc Y\times_\mc X \mc Y)}(r^* A, r^* B)$.
\end{T}

\begin{proof} From the projective bundle axiom it follows each ${p^i}^*$ is injective on the level of automorphism-groups, \ie for any object $A$ in $P(\mc Y_i)$,
$\op{Aut}_{P(\mc Y_i)}(A) \to \op{Aut}_{\mc Y_{i+1}}({p^i}^* A)$ is injective,
so the functor is faithful. For
(b), the condition is obviously necessary. To prove that the condition is sufficient we can assume $A = B$. Let $0 = E_0
\subseteq E_1 \subseteq \ldots \subseteq  E_e = p^* E$ be
the universal flag on $\mc Y$, and $L_i = E_i/E_{i-1}$, then by the projective bundle axiom
we have natural isomorphisms
$$\op{Aut}_{P(\mc Y)}(A) = \bigoplus_{j_1 = 0}^{e} \ldots \bigoplus_{j_{e} = 0}^1  L_1^{j_1} \otimes \ldots \otimes L_e^{j_e} \otimes p^* \op{Aut}_{P(\mc X)}(A)$$
and
$$\op{Aut}_{P(\mc Y \times_\mc X \mc Y)}(A) = \bigoplus_{j_1, j'_1 = 0}^{e} \ldots \bigoplus_{j_{e} = 0, j'_e = 0}^1  p_1^* L_1^{j_1} \otimes
p_2^* L_1^{j'_1} \otimes \ldots \otimes p_1^* L_e^{j_e} \otimes p_2^*
L_e^{j'_e} \otimes r^* \op{Aut}_{P(\mc X)}(A).$$
Representing $f$ in the form suggested above, we see that $p_1^*
(f) = p_2^* (f)$ exactly when all components of $f$ are zero
except for the one belonging to $(j_1, \ldots, j_e) = (0, \ldots,
0)$, which means exactly that $f$ is equal to  $p^* h$ for some
morphism $h: A \to A$. Moreover $h$ is unique because of (a).
\end{proof}

By basechange to the flag-variety we can suppose we have nice
enough flags. If we define an isomorphism dependent on this flag,
the content of the proposition is that this descends to the base
whenever this isomorphism isn't dependent on the flag.

\subsection{Adams and $\lambda$-operations on the virtual category}
Let $S$ be a scheme, and $\mc X$ an algebraic stack over $S$. Recall that we denote by $\mathbf{P}(\mc X)$ the category of vector bundles on $\mc X$. Denote by $V(\mc X)$ the virtual category thereof. We have the following result
which is more or less contained in \cite{Grayson};

\begin{Prop} \label{Prop:Adamsoperation} There is a unique family of determinant functors $\Psi^k: \mathbf{P}(\mc X) \to
V(\mc X)$, and thus $\Psi^k: V(\mc X) \to V(\mc X)$, stable under
pullback, such that

\begin{itemize}
  \item If $L$ is a line bundle, $\Psi^k(L) =
  L^{\otimes k}$. \\
  \item $\Psi^k \circ \Psi^{k'} \simeq \Psi^{k k'}$.
\end{itemize}
\end{Prop}
\begin{proof} Unicity of the operations clearly follows from the characterizing properties and
the splitting principle  (Theorem \ref{thm:splittingprinciple}). To prove existence, we apply the ideas of
loc.cit.. Let $N$ be a complex of vector bundles,
and $CN$ be the cone of the identity morphism $\op{id}: N \to N$.
Furthermore, let $S^k$ be the $k$-th symmetric power, so that the
$p$-th term of $S^k CN$ is $S^{k-p} N\otimes \wedge^p N$, whenever $N$ is reduced to a vector
bundle in degree 0 (for details, see
loc.cit., p. 4). Finally, for a bounded complex $N_\bullet = [\ldots
\to N_{i-1} \to N_{i} \to N_{i+1} \to \ldots]$, define the secondary
Euler characteristic $\chi'(N_\bullet) = \sum (-1)^{p+1} p [N_p] \in
V(\mc X)$. One of the key ideas of loc.cit. (formula (3.1)) is the formula
in $K_0(\mc X)$, for a vector bundle $E$,
$$\Psi^k(E) = \chi'(S^k CE).$$ We propose the same definition for Adams operations in the virtual category $V(\mc X)$.
Clearly $\Psi^k(L) = L^{\otimes k}$ for a line bundle $L$. Now, given a flag $E_1 \subseteq E_2 \subseteq \ldots \subseteq E_n$, define
 $E_1 \cdot E_2 \ldots \cdot E_n$ to be the image of $E_1 \otimes E_2 \otimes \ldots \otimes E_n$ in $S^n E_n$. Suppose that we have an exact
sequence of vector bundles $0 \to E' \to E \to E''
\to 0$, and consider the filtration
\begin{eqnarray*} S^k CE'  & = & CE' . CE' . \ldots . CE' . CE' \\
& \subseteq & CE' . CE' . \ldots . CE' . CE \\
& \subseteq & CE' . CE' . \ldots . CE . CE  \subseteq \ldots \\
& \subseteq & CE . CE . \ldots . CE . CE = S^k CE
\end{eqnarray*}
induces by a certain addivity of the secondary Euler characteristic,
isomorphisms \begin{eqnarray*} \chi'(S^k CE) & = & \chi'(S^k CE'') +
\chi'(S^k CE') + \sum_{i = 0}^{k-1} \chi'(S^{i}CE'' \otimes S^{k-i} CE') \\
& = &  \chi'(S^k CE'') + \chi'(S^k CE') \end{eqnarray*} since the
secondary Euler-characteristic of a product of acyclic complexes is
0 and by the multilinearity-property of loc.cit (formula (2.1)). We need only
verify that this operation respects filtrations. Let $F \subseteq H
\subseteq E $ be an admissible filtration, and consider the double
graded filtrations of $S^k C E$, $A_{\bullet, \bullet}$, where
$A_{i,j} = S^{k-i-j}CE . S^{j}CF . S^{i}CH.$ Applying
secondary Euler characteristics in every direction, we obtain that
the diagram of isomorphisms
$$\xymatrix{\chi'(S^k CE) \ar[d] \ar[r] & \chi'(S^k CH) + \chi'(S^k CE/H) \ar[d] \\
\chi'(S^k CF) + \chi'(S^k CE/F) \ar[r] & \chi'(S^k CF) + \chi'(S^k
CH/F) + \chi'(S^k CE/H)}$$ constructed above commutes. Condition "b)" of Definition \ref{defn:determinantfunctor} is trivial. Also everything is
clearly stable under pullback. The last point now follows by unicity.
\end{proof}

\begin{remark} In the next chapter we will show that whenever we restrict ourselves to regular schemes, the constructed Adams-operations are actually unique, at least whenever one inverts 2 or more primes in the virtual category.
\end{remark}

We record the following corollary (of the splitting principle
applied to the above case):
\begin{Cor} \label{Cor:adams-multiplicative} $\Psi^k: V(\mc X) \to V(\mc X)$ is a ring-homomorphism in
the sense that there are natural isomorphisms, for $A, B \in \rm{ob}
V(X)$,
$$\Psi^k(A \otimes B) \simeq \Psi^k(A) \otimes \Psi^k(B)$$
compatible with the above sum-operation and compatible with basechange.
\end{Cor}

\begin{proof} We only need to verify the multiplicative operation. It suffices to show that for any
$A \in V(\mc X), B \in \mathbf{P}(\mc X)$, $\Psi^k(A \otimes B) = \Psi^k(A) \otimes \Psi^k B)$ naturally.
Or, by the splitting principle since $\Psi^k$ is already an additive determinant functor, that if $B$ is a line bundle, then $\Psi^k(A \otimes L) = \Psi^k(A) \otimes L^{\otimes k}$ naturally. For this we can assume that $A$ is also a line-bundle $M$, in which case we have
$\Psi^k(M \otimes L) = (M \otimes L)^{\otimes k} = M^{\otimes k} \otimes L^{\otimes k} = \Psi^k(M) \otimes \Psi^k(L)$.
\end{proof}

The method of \cite{IntviaAdams} also provides us with $\lambda$-operations:
\begin{Prop} \label{prop:lambdasoule} Let $\mc X$ be an algebraic stack and $\mc Z$ a close substack thereof and let $k$ be a positive integer. There are naturally defined functors $\lambda^k: V^\mc Z(\mc X) \to V^\mc Z(\mc X)$ satisfying the following compatibilities:
\begin{enumerate}
  \item $\lambda^1 = \op{id}$.
  \item They are stable under basechange.
  \item If $\mc Z = \mc X$ and $E$ is a vector bundle, $\lambda^k E = \wedge^k E$.
  \item If $$0 \to E'_\bullet \to E_\bullet \to E''_\bullet \to 0$$ is a short exact sequence of complexes of vector bundles exact off $\mc Z$ there is a canonical isomorphism, where we set $\lambda^0 = 1$:
      $$\lambda^k (E_\bullet) = \sum_{i+j= n} \lambda^i E'_\bullet \otimes \lambda^j E''_\bullet.$$
\end{enumerate}
\end{Prop}
\begin{proof} The method of \loccit provides us with a functor, via the Dold-Puppe construction, a functor $\lambda^k: \mathbf{P}_\mc Z(\mc X) \to \mathbf{P}_\mc Z(\mc X)$ where $\mathbf{P}_\mc Z(\mc X)$ denotes the category of complexes of vector bundles on $\mc X$ exact off $\mc Z$. Thus a functor $\lambda^k: (\mathbf{P}_\mc Z(\mc X), q.i.) \to V^\mc Z(\mc X)$. Writing $V^\mc Z(\mc X)[[t]] := \coprod_{k \geq 0 } V^\mc Z(\mc X) t^k$ it is naturally a Picard category with respect to multiplication with unit element $1 \otimes t^0 + \sum 0 \otimes t^k$. The usual proof shows that $\lambda_t = \sum_{k \geq 0} \lambda^k t^k$ is a determinant functor $\lambda_t: (\mathbf{P}_\mc Z(\mc X), q.i.) \to V^\mc Z(\mc X)[[t]]$ and hence the searched for functor. Notice that we use Theorem \ref{thm: waldhausen-determinant} since the category in question is not an exact category with isomorphisms but a Waldhausen category with quasi-isomorphisms.
\end{proof}
\begin{Cor}\label{Cor: adams-lambda-relation} Suppose that $Z = X$. Then there is a canonical isomorphism
$$\Psi^k = \lambda^1 \otimes \Psi^{k-1} - \lambda^2 \otimes \Psi^{k-2} + \ldots + (-1)^{k} \lambda^{k-1} \otimes \Psi^1 + (-1)^{k+1} k \lambda^k$$
where $\Psi^*$ and $\lambda^*$ denotes the functors constructed in Proposition \ref{Prop:Adamsoperation} and Proposition \ref{prop:lambdasoule}.
\end{Cor}
\begin{proof} The left hand side is already an endofunctor on the virtual category of vector bundles. We need to verify that the right hand side is additive since by the splitting principle we can then reduce to the case of line bundles, for which the statement is clear. This follows by induction on $k$ and the property (d) in the above proposition.
\end{proof}
This also defines inductively Adams operations via $\lambda$ operations (cf. Proposition \ref{Prop:Iversen} below) to virtual categories with support. \\ The following is a consequence of the calculation on the level of complexes in \loccit:
\begin{Cor} \label{Cor:adamslocalring}  Let $R$ be a ring with $a$ a non-zero divisor of $R$. Denote by $K(a)$ the complex $[R \stackrel{a}\to R]$ with $R$ placed in degree 0 and 1 and $X = \spec R, Y = \spec R/a$. Then there is a canonical isomorphism $\Psi^k(K(a)) \simeq k K(a)$ in $V^{Y}(X)$.

\end{Cor}
\begin{proof} The alluded to calculation shows that that $\lambda^k[K(a)] = [K(a)][1-k]$ on the level of complexes. The result follows from Corollary \ref{Cor: adams-lambda-relation} noting that $[K(a)]^{\otimes 2} = 0$ naturally.
\end{proof}
\begin{Prop} \label{Prop:Iversen} The Adams operations uniquely extend via Corollary \ref{Cor: adams-lambda-relation} to operations on virtual categories with support such that they are stable under pullback and compatible with the functor $\Phi_Y^X:V^Y(X) \to V(X)$ for a scheme $X$ with $i: Y \hookrightarrow X$ a closed subscheme.
\end{Prop}
\begin{proof} By the general arguments of \cite{Iversen}, any complex $E_\bullet$ on a scheme $X$ which is acyclic outside of a closed subscheme $Y$ pulls back from a universal complex $C_\bullet$ on a classifying-type scheme $\pi: G \to X$ acyclic outside of $G_Y$, such that the support of $C_\bullet$ maps to the support of $E_\bullet$. Moreover, this scheme has the property that $V^{G_Y}(G) \to V(G)$ is faithful. The functor $\Omega: \Phi_{G_Y}^G \lambda_t(\mc E_\bullet) \to \lambda_t(\Phi_{G_Y}^G \mc E_\bullet)$ where $E_\bullet$ is any complex of vector bundles on $G$ with support on $G_Y$ and $\lambda_t$ is the functor in the proof of Proposition \ref{prop:lambdasoule}. Restricting this isomorphism to $G \setminus G_Y$ gives both sides canonically isomorphic to zero compatible with the identity map of the zero-object. By the exact sequence $0 \to \pi_1(V^{G_Y}(G)) \to \pi_1(V(G)) \to \pi_1(V(G \setminus G_Y)) \to 0$ this restriction comes from a canonical element in $\pi_1(V^{G_Y}(G))$. This argument proves that any functor $\lambda_t$ on $V^Y(X)$ compatible with pullback and $\Phi_Y^X$ is unique up to unique isomorphism.
\end{proof}
The following will be used to describe the functorial filtration on the virtual category exhibited in the next chapter in terms in terms of a filtration by dimension.
\begin{Cor} \label{Cor:adams-K_1} Let $R$ be a regular ring of dimension $d$ and let $\mathfrak m$ be a maximal ideal given by a $R$-regular sequence $(a_1, \ldots, a_d)$. Write $X = \spec R, Y = \spec R/\mathfrak m$. Then $\Psi^k$ acts on $K_1^{Y}(X)$ by multiplication by $k^{d+1}$.
\end{Cor}
\begin{proof} In general we can replace $R$ by the localization $R_\mathfrak m$ and suppose also first that $d = 1$ so that $R$ is a discrete valuation ring. Then the the determinant functor associates to a complex of vector bundles on $X$ acyclic outside of $Y$ a line bundle with a rational section with poles on $Y$. Defining an isomorphism of two such sections as an isomorphism of line bundles interchanging the sections and identifying two isomorphisms that differ by an element of $1 + \mathfrak m$ this induces an equivalence of categories (both sides have $\pi_0 = \bb Z$ and $\pi_1 = k(Y)^*$ respectively), and the determinant functor commutes with the Adams operations up to equivalence. It suffices to evaluate $\Psi^k$ on an automorphism of a generator of the isomorphism classes of this category, which can be taken as the canonical section represented by $\calo(-Y) \hookrightarrow \calo_X$. Then $\Psi^k$ acts on the automorphisms of line bundles by multiplication by $k$. The induced automorphism determined by an automorphism of the section can be calculated by factoring $\calo(-kY) \to \calo_X$ as  $\calo(-kY) \hookrightarrow \calo((-k+1)Y) \hookrightarrow \ldots \hookrightarrow \calo(-Y) \hookrightarrow \calo_X$. In the general case, since $(a_1, \ldots, a_d)$ is a $R$-regular sequence the complex $\op{Kos}(a_1, \ldots, a_d) = \bigotimes_{i=1}^d \op{Kos} (a_i)$, where $\op{Kos}(a_i)$ is the associated Koszul complex to $a_i$, is a resolution of $Y$ and generates $\pi_0(V^Y(X)) = \bb Z$. Consider an automorphism of this object. By restriction along subschemes defined by $a_1, \ldots, a_{d-1}$ we get an quasi-isomorphism of the image of $\op{Kos}(a_1, \ldots, a_d)$ in $R/(a_1, \ldots, a_{d-1})$, and thus a quasi-isomorphism of $\op{Kos}(a_d)$ in $R/(a_1, \ldots, a_{d-1})$ which induces every automorphism of $V^Y(X)$ and thus we only have to consider the automorphism of an object $a_d$ in $V^Y(X)$. By the case $d=1$ the functor $\Psi^k$ acts on this automorphism by $k^2$ times $\Psi^k([\op{Kos}(a_1, \ldots, a_{d-1})]) = \prod_{i=1}^{d-1} \Psi^k([\op{Kos}(a_i)]) = k^{d-1}[\op{Kos}(a_1, \ldots, a_{d-1})]$ by Corollary \ref{Cor:adamslocalring} and we conclude. \end{proof}
Using this and copying the rest of the proof of \cite{Soule-Arakelov}, Chapter I, Lemma 6, we now obtain:
\begin{Cor} \label{Cor:adams-weight-relation} Let $X$ be a regular scheme and $Y$ a closed subscheme of codimension $m$. Define
$$F^i K_1^Y(X) = \bigcup_{Z \subset Y, \op{codim_X Z} \geq p} \op{Im} [K_1^Z(X)_\bb Q \to K_1^Y(X)_\bb Q]$$
where the union is over all closed subschemes of codimension at least $i$ in $X$ with support in $Y$. Then
$$F^i K_1^Y(X) = \bigoplus_{p \geq i} K_1^Y(X)^{(p+1)}$$
where $K_1^Y(X)^{(p)}$ denotes the $p$-th Adams eigenspace of $K_1^Y(X)_\bb Q$, \ie where $\Psi^k$ acts by multiplication by $k^p$.
\end{Cor}
We conclude this section by considering the following functor \begin{equation*} \label{lambda-1} \lambda_{-1}: (\mathbf{P}({\mc X}), \op{iso}) \to V({\mc X}) \end{equation*}
defined as follows. If $E$ is a vector-bundle on $\mc X$, we define $\lambda_{-1}(E)$ as the alternating product of exterior powers $\sum_{i = 0}^\infty (-1)^i \wedge^i E$. This is an object which is unique up to canonical isomorphism. Given a short exact sequence of vector-bundles
$$0 \to  F \stackrel{\pi}{\to}  E \to  E/F \to 0$$
we can define an isomorphism
\begin{equation} \label{lambda-2} \lambda_{-1}(E) \to \lambda_{-1}(F) \otimes \lambda_{-1}(E/F) \end{equation}
as follows; for an integer $k$, such that $0 \leq k \leq n = \op{rk} F$,
we have a well known natural filtration (cf. \cite{SGA6}, Ch. V, Lemme 2.2.1) of $\wedge^k F$ whose $i$-th instance is given by
$F^i \wedge^k E = \op{Im}[\wedge^i F \otimes \wedge^{k-i} E \to \wedge^k E]$, and with successive quotients $\wedge^i F \otimes \wedge^{k-i} F/E$,
thus giving isomorphisms
\begin{equation} \label{wedgeoperation} \wedge^k E \simeq \sum_{i = 0}^k \wedge^i F \otimes \wedge^{k-i} E/F.
\end{equation}
Now, given two virtual vector bundles $A$ and $B$, we have $$A \otimes B + (-A) \otimes B = (A + (- A)) \otimes B = 0 \otimes B = 0$$ and thus
an isomorphism $(-A)\otimes B) \simeq -(A \otimes B)$. Analogously we obtain $A \otimes (-B) \simeq -(A\otimes B)$. The diagram
$$\xymatrix{(-A) \otimes (-B) \ar[r] \ar[d] & -(A \otimes (-B)) \ar[d] \\
-((-A)\otimes B) \ar[r] & --(A \otimes B) = A \otimes B}$$
is only commutative up to sign $\epsilon(A \otimes B)$ (cf. \cite{determinant}, 4.11 a) + b)).
We define the isomorphism $(-1)^k \wedge^k E = \sum_{i = 0}^k [(-1)^i \wedge^i F] \otimes [(-1)^{k-i}\wedge^{k-i} E/F]$
via (\ref{wedgeoperation}), the isomorphism
\begin{eqnarray*}
(-1)^k \wedge^i F \otimes \wedge^{k-i} E/F & = & (-1)^{k-i} (((-1)^i \wedge^i F) \otimes \wedge^{k-i} E/F) \\ & = & ((-1)^i \wedge^i F) \otimes (-1)^{k-i}\wedge^{k-i} E/F .
\end{eqnarray*}
The isomorphisms (\ref{lambda-2}) and (\ref{wedgeoperation}) are easily verified to be compatible (up to sign) with successive admissible filtrations $F' \subset F \subset E$
by considering the double filtration $F^{i,j} F = \op{Im}[\wedge^i F' \otimes \wedge^{j} F \otimes \wedge^{k-i-j} E \to \wedge^k E]$. \mps{actually, I should check what the sign is. The other part is obvious. In any case, when ignoring signs there are no issues, which is the case we employ}

These operations are however only commutative and associative up to a nightmare of signs, but become commutative and associative once one get rid of these.

\section{Rigidity and operations on virtual categories} \label{section:rigidity} 
In this section we exhibit certain rigidity-properties of the virtual categories we are working on, and also the main technical results of this
part of the thesis. As such, it rests heavily on the results obtained in $\cite{Morel}, \cite{Riou}$ and $\cite{VoevMorel}$, and can perhaps
in many instances be seen as reformulations of results therein obtained. The formulation in terms of $K$-cohomology was inspired from \cite{Toen}.

\subsection{The main result on rigidity} The main result of this section (Theorem \ref{thm:rigidity}) can be phrased, in a certain situation, that there is a certain commutative diagram
$$\xymatrix{hom_{\mc H(\mathfrak R)}(\mc K_\bb Q, \mc K_\bb Q) \ar@{=}[rr] \ar[dr] & & \hom_{\mathfrak R^{op} \op{Set}}(K_0(-)_\bb Q, K_0(-)_\bb Q) \ar@/{}_{1pc}/[dl] \\
 & \hom(V_\bb Q, V_\bb Q) \ar@/{}_{1pc}/[ur]  }.$$
Here $\hom_{\mathfrak R^{op} \op{Set}}(K_0(-)_\bb Q, K_0(-)_\bb Q)$ is a set of natural endo-transformations of the presheaf $K_0(-)$ on the category
of regular schemes, and $\hom(V_\bb Q, V_\bb Q)$ is the set of endo-functors of the virtual category of algebraic vector bundles strictly stable under pullback. Finally,
$\mc K_\bb Q$ is a simplicial sheaf representing (rational) algebraic $K$-theory. This allows us to associate functorial operations on $V_\bb Q$ via the corresponding operations on $K_0$. We refer to the theorem for
a precise formulation. \\
Let $X$ be a separated regular Noetherian scheme of finite Krull dimension
$d$. Then it is well known (see for example \cite{Riemann-RochAlgebra}, chapter V, Corollary 3.10, \cite{SGA6}, chapter VI, Théorème 6.9 or use \cite{Thomason2}, Theorem 7.6 and (10.3.2))
then any element $x$ of $K_0(X)$ of virtual rank 0 is nilpotent, and moreover we have
$x^{d+1} = 0$. One can prove this in several ways, but one of the
most natural ways is to construct a certain filtration on $K_0(X)$
which can be compared to other filtrations in terms of dimension of
supports, a filtration that will terminate for natural reasons (see loc.cit.). One such
filtration is the $\gamma$-filtration $F^p_\gamma$, built out of the
$\lambda$-ring structure on $K_0(X)$ (see \cite{Riemann-RochAlgebra}, chapter III, p. 48 or \cite{SGA6}, chapter V, 3.10). We wish to incarnate this
kind of nilpotence in the virtual category of $X$. Obviously, if $x$
is a virtual vector-bundle of rank 0, then we know that a high
enough power of it is isomorphic to a zero-object, but only
non-canonically. A naive idea is to search for a decreasing filtration
$\op{Fil}^i$ of $V(X)$ which has the property that the
functors $\op{Fil}^i \to \op{Fil}^{i-1}$ are faithful
additive functors, and for big enough $p$, $\op{Fil}^p$ is a category with
exactly one morphism between any two objects. \\
The approach we have chosen to the problem is to construct the
filtration already on the level of classifying spaces of the $\bb P^1$-spectrum representing rational algebraic $K$-theory in $\mc {SH}(S)$,
and then use simplicial realizations to obtain a canonical filtration of $BQ{\mathbf{P}(X)}$ which eventually terminates or becomes trivial.
For the notation used in this section we refer the reader to the Appendix \ref{B}. \\
 Grayson (see \cite{Grayson-weight}) proposes that there should be a multiplicative filtration  $W^i$
of a space $K(X)$ representing the $K$-theory of $X$;
$$\ldots \to W^2 \to W^1 \to W^0 =  K(X)$$
such that the two following properties are satisfied:
\begin{enumerate}
  \item For any $t$, the quotient $W^i/W^{i+1}$ is the simplicial realization of a simplicial abelian group.
  \item The Adams operations $\Psi^k$ act by multiplication by $k^i$ on $W^i/W^{i+1}$.
\end{enumerate}
Such a filtration would immediately give an exact couple and thus
give rise to an Atiyah-Hirzebruch spectral sequence
$$E_2^{p,q} = H^{p-q}(X, \bb Z(-q)) \Rightarrow K_{-p-q}(X)$$
relating "motivic cohomology" (that is, cohomology of \linebreak $\bb Z(i) = \bb Z(i)^W := \Omega^{2i}(W^i/W^{i+1})$, in the sense of spectra with negative homotopygroups) on the left with algebraic $K$-theory on the right. In \loccit it is noted that the Postnikoff filtration satisfies the first but not the second property. For smooth schemes over a field \cite{Levine} constructed a coniveau-filtration which gives the correct spectral sequence for smooth varieties over a field.  \\
The starting point of this section is the following theorem, which states that if we tensor with $\bb Q$ we can construct a Grayson-like filtration with various functorial properties. The author ignores if the filtration of \cite{Levine} coincides with the one considered in this section, both considered as objects of the appropriate homotopy category of schemes.
\begin{T} \label{thm:BGLfil} There are $H$-groups $\{\op{Fil}^{(i)}\}_{i=0}^\infty$ (\ie group-objects) and \linebreak $\{\bb{H}^{(i)}\}_{i=0}^\infty$ of $\mc H(\mathfrak R_S)_\bullet$, determined up to unique isomorphism, satisfying the following properties:
\begin{enumerate}
  \item $\op{Fil}^0 = (\bb Z \times \op{Gr})_\bb Q$ and for any $i \geq 0$, there are morphisms $\op{Fil}^{(i+1)} \to \op{Fil}^{(i)}$.
  \item For any $i,j$, there are natural pairings $\op{Fil}^{(i)} \wedge \op{Fil}^{(j)} \to \op{Fil}^{(i + j)}$ making, for $i' \leq i, j' \leq j$, the following diagram commutes $$\xymatrix{\op{Fil}^{(i)} \wedge \op{Fil}^{(j)} \ar[r] \ar[d] & \op{Fil}^{(i + j)} \ar[d] \\
      \op{Fil}^{(i')} \wedge \op{Fil}^{(j')} \ar[r] & \op{Fil}^{(i' + j')} }.$$
  \item For any $i,j$, there are natural pairings $\bb H^{(i)} \wedge \bb H^{(j)} \to \bb H^{(i+j)}$.
  \item There is a factorization $\op{Fil}^{(i+1)} \to \op{Fil}^{(i+1)} \times \bb H^{(i)} \stackrel{\Phi_i}\thickapprox \op{Fil}^{(i)}$ which is compatible with the two above products.
  The pairings are also associative in the obvious sense.
  \item The Adams operations $\Psi^k$ act on all the above objects and morphisms and acts purely by multiplication by $k^i$ on $\bb H^{(i)}$.
\end{enumerate}
\end{T}

\begin{proof} It follows from Theorem \ref{thm:BGLdecomposition} that we have a filtration of $\mathbf{BGL}_\bb Q$ in $\mc {SH}(\mathfrak R_S)$ given
by $\op{Fil}^p = \bigoplus_{n \geq p} \bb H^{n}$. By definition
there is an evaluation-functor $\operatorname{ev}_n: \mc {SH}(\mathfrak R_S) \to \mc H(\mathfrak R_S)_\bullet$ sending a spectra (cf. Appendix \ref{B})
$\mathbf{E}$ to $\mathbf{E}_n$. Evaluating at 0 we obtain a
canonical filtration of $\op{ev}_0(\mathbf{BGL}_\bb Q) \simeq (\bb Z
\times \op{Gr})_\bb Q$ (see Definition \ref{defn:stablemodel}), a
filtration $\{\op{Fil}^{(i)}\}_{i=0}^\infty$ in $\mc H(\mathfrak R_S)_\bullet$. We similarly define $\bb H^{(i)} = \op{ev}_0 (\bb H^i)$ so that
$\op{Fil}^{(i)} = \bb H^{(i)} \oplus \op{Fil}^{(i-1)}.$ They
are the 0-th space of a $\bb P^1$-spectrum and automatically $H$-groups. We similarly define Adams operations $\Psi^k$ on the various objects
via the same functor $\op{ev}_0$. We need to verify the other claimed properties.

Let $\mc X$ be a pointed simplicial sheaf, and define \linebreak $\Omega^j X =
\uHom_{\Delta^{op} \op{Shv}_\bullet(\mathfrak{R}_S)}(S^j, \mc X )$, the right adjoint to $S^j
\wedge -$. Also denote by $R\Omega^j$ the total derived functor of
$\Omega^j$ in $\mc H(\mathfrak R_S)_\bullet$. \\

Denote by $\mathfrak R_S^c$ the category whose objects are open inclusions of regular schemes $j: U \to X$. Recall that the Yoneda functor $\Phi$ is defined by
$$\Phi: \mathfrak R_S \to \Delta^{op}\op{Shv}(\mathfrak R_{S, \op{sm}}) \to \mc H (\mathfrak R_S)$$
and for any object $G \in \mc H(\mathfrak R_S)$ we denote by $\phi
(G)$ the presheaf
$$\mathfrak R_S \ni U \mapsto \Hom_{\mc H(\mathfrak R_S)}(\Phi U, G)$$
as follows. We define the quotient $\Phi (U \to X) := \Phi X/\Phi (X-U)$ in $\mc H(\mathfrak{R}_S)$ and we set $\phi(G)(U \to X) = \Hom_{\mc H(\mathfrak R_S)}(\Phi (U \to X), G)$. The following follows from the general theory in \cite{Thomason2}:
\begin{Prop} Let $U \to X$ be an open inclusion of regular schemes. The homotopy fiber of
$$\uHom_{\Delta^{op} \op{Shv}_\bullet(\mathfrak{R}_S)}(X,  \bb Z \times \op{Gr} ) \to \uHom_{\Delta^{op} \op{Shv}_\bullet(\mathfrak{R}_S)}(U,  \bb Z \times \op{Gr} )$$
identifies with the $K$-theory space of finite complexes of vector bundles on $X$ exact on $U$.
\end{Prop}
We then have:
\begin{TL} \label{TL:phi-virtual} Let $j \geq 0$. We have the following natural isomorphisms of presheaves on $\mathfrak R_S^c$, where we set $Z := X \setminus U$,  in the following cases:
\begin{itemize}
  \item $\phi(R\Omega^j (\bb Z \times \op{Gr}))(U \to X) = K_j^Z(X)$. This also holds for localizations by integers $n$ and $\bb Q$.
  \item $\phi(R\Omega^j \bb H^{(i)})(U \to X) = K_j^Z(X)^{(i)}$, the presheaf of sections of $(U \to X) \mapsto K_j^Z (X)_\bb Q$ with $\Psi^k$-eigenvalue $k^i$ (which is independent of $k \geq 2$).
  \item $\phi(\op{Fil}^{(i)})(U \to X) = F^iK_0^Z (X)_\bb Q = \bigoplus_{p \geq i} K_0^Z(X)^{(p)}$, where $$F^iK_0^Z(X)_\bb Q = \bigcup_{Y \subset X}
  \op{im}[K^{Z \cap Y}_0(X)_\bb Q \to K_0^Z(X)_\bb Q]$$ and the union is over all closed subschemes $Z \subset X$ of codimension at least $i$.
  \item Suppose $U = \emptyset$. Let $\bb P^\infty = \op{colim}_n \bb P^n$, then $\phi(\bb P^\infty)(-) = \Pic (-)$, $\phi(R\Omega \bb P^\infty) = \bb G_m$ and $\phi(R\Omega^j \bb P^\infty)=0$ otherwise.
\end{itemize}
\end{TL}
\begin{proof} In view of how the Adams-operations act on the various objects involved, using Theorem \ref{thm:voevmorelGr} the
first non-trivial part is the equality \linebreak $ \bigoplus_{p \geq i} K_0(-)^{(p)}  =  F^iK_0(-)_\bb Q$
which is \cite{Soule-Arakelov}, chapter I, Lemma 6 \mps{I have been unable to find any reference of this in \cite{SGA6}.}. The last point is established as in \cite{VoevMorel}, Section 4, Proposition 3.8.
\end{proof}

\begin{TL} \label{Prop:presheaftohomotopy} [\cite{Riou}, proof of Théorème III.29 + 31, ] Suppose $S$ is a regular scheme, and $\mc X$ and $\mc Y$ are objects of $\mc H(\mathfrak R_S)_\bullet$. Then the natural maps
$$\Hom_{\mc H(\mathfrak R_S)_\bullet}(\mc X, \mc Y) \to \Hom_{\bullet, \mathfrak R_S^{op} \op{Set}}(\phi \mc X, \phi \mc Y)$$
and
$$\Hom_{\mc H(\mathfrak R_S)}(\mc X, \mc Y) \to \Hom_{\mathfrak R_S^{op} \op{Set}}(\phi \mc X, \phi \mc Y)$$
are bijective in the case $\mc Y$ and $\mc X$ are products of any of the following:
\begin{itemize}
  \item $(\bb Z \times \op{Gr})$ or any localization thereof by natural integers $n$ or $\bb Q$.
  \item $\op{Fil}^{(i)}.$
  \item $\bb H^{(i)}$.
  \item $\bb P^\infty$.
\end{itemize}
\end{TL}
\begin{proof} The cited proof goes through with the following remarks. By ibid, Lemme III.19, for any objects $X$ and $E$ in $\mc H(\mathfrak R_S)_\bullet$
with $E$ an $H$-group, there is an injection $\Hom_{\mc H(\mathfrak R_S)_\bullet}(X,E) \to \Hom_{\mc H(\mathfrak R_S)}(X,E)$ whose image is that of morphisms
$X \stackrel{f}\to E$ such that $$f^*(\bullet) = \bullet \in \Hom_{\mc H(\mathfrak R_S)}(S, E).$$ Thus one reduces to the non-pointed case. The objects in question are retracts of $(\bb Z \times \op{Gr})_\bb Q$ or equal to $\bb P^\infty$, which are the cases treated in the reference and one concludes.
\end{proof}

We are now ready to complete the proof of Theorem \ref{thm:BGLfil}. Using
the above two lemmas we deduce morphisms $\op{Fil}^{(i)} \times \op{Fil}^{(j)} \to \op{Fil}^{(i+j)}$ from the morphisms
$F^iK_0(-) \times F^jK_0(-) \to F^{i+j}K_0(-)$ and similarly for $\bb H^{(i)} \times \bb H^{(j)} \to \bb H^{(i+j)}$. As in ibid, Lemma III.33 \mps{no proof is made
of this assertion though} we have the following proposition:
for an $H$-group $E$ and objects $A,B$ of $\mc H(\mathfrak R_S)_\bullet$, the
map $$\Hom_{\mc H(\mathfrak R_S)_\bullet}(A\wedge B, E) \to \Hom_{\mc H(\mathfrak R_S)_\bullet}(A \times B, E)$$
is injective and its image consists of morphisms $A \times B \to E$ such that the restriction to $\bullet \times B$ and $A \times \bullet$ is zero.
It follows that both of the two morphisms factor as $\op{Fil}^{(i)} \wedge \op{Fil}^{(j)} \to \op{Fil}^{(i+j)}$ and $\bb H^{(i)} \wedge \bb H^{(j)} \to \bb H^{(i+j)}$. The same argument
shows the necessary diagrams are commutative. The Adams-operations act appropriately for the same reason.
\end{proof}

For the below, recall that $V_\mc X$ is the associated virtual category to $\mc X$ as in Proposition \ref{Prop:virtualhomotopy}. The main theorem of this section is now the following:
\begin{T} [Rigidity] [proof of \cite{Riou}, Section III.10] \label{thm:rigidity} Suppose $S = \spec \bb Z$. In the cases considered in the above lemma, except whenever $\mc Y$ involves a factor of $\bb P^\infty$,
the morphisms
$$\Hom_f(V_\mc X,  V_\mc Y) \to \Hom_{\mathfrak R^{op} \op{Set}}(\phi \mc X, \phi \mc Y) $$
and
$$\Hom_{f,\bullet}(V_\mc X,  V_\mc Y) \to \Hom_{\bullet, \mathfrak R^{op} \op{Set}}(\phi \mc X, \phi \mc Y)$$
have natural sections (and are thus surjective), which are
canonical up to unique isomorphism (see \ref{defn:V_X} for a
definition of the functor $V$). Moreover, any natural transformation of two
functors $V_\mc X \to V_\mc Y$ is unique up to unique natural transformation.
\end{T}

\begin{proof} This follows directly from Lemma \ref{Prop:presheaftohomotopy} and by considering the composition
$$\Hom_{\mc H(\mathfrak R_S)}(\mc X, \mc Y) \to \Hom_f(V_\mc X,  V_\mc Y) \to \Hom_{\mathfrak R^{op} \op{Set}}(\phi \mc X, \phi \mc Y)$$
obtained from pre-rigidity in Proposition
\ref{Prop:virtualhomotopy}. The essential point is that
$\Hom_{\mc H(\mathfrak R_S)}(\mc X, \Omega \mc Y)$ disappears since
they can all be related to $K_1(\bb Z) = \bb Z/2$-modules, and all
the groups in question are $2$-divisible by construction. Now,
suppose we have $\phi  \in \Hom_f(V_\mc X,  V_\mc Y)$ and an
automorphism $\phi$, \ie a functorial isomorphism of functors
$\phi \simeq \phi$. Suppose for simplicity that $\mc X = \mc Y =
(\bb Z \times \op{Gr})$. It is easy to see it determines an element
in $\Hom_{\mathfrak R^{op} \op{Set}}(K_0(-)_\bb Q, K_1(-)_\bb Q)$,
and moreover that any such element determines an automorphism of
$\phi$. The latter group is zero by Theorem \ref{thm:voevmorelGr}
and an argument analogous to the proof in the previous lemma.
\end{proof}
Under the conclusion of the above theorem we say that the functor $V_\mc X \to V_\mc Y$ lifts that of $\phi \mc X \to \phi \mc Y$. We say the lifting given by the section of the theorem is given by "rigidity".

To state the next proposition, denote by $\mathfrak{Pic}\frac{1}{n}(\mc X)$ and $\mathfrak{Pic}_\bb Q(\mc X)$ the Picard category of line bundles on $\mc X$
localized at an integer $n$ or $\bb Q$ respectively. 

\begin{Prop} Let $\mathfrak R \mathfrak{Ch}/S$ be the category of regular algebraic stacks over $S$, and $\mathfrak R \mathfrak{Ch}^c/S$ the category of inclusions of regular algebraic stacks $U \to X$. Put $\Phi':  \mathfrak R \mathfrak{Ch}/S \to \mc H(\mathfrak R)$
be the functor determined by the extended Yoneda-functor (see Definition \ref{defn:extyoneda}) and for an object $\mc X$ of $\mc H(\mathfrak R)$,
denote by $\phi'(\mc X)$ the functor $\mathfrak R \mathfrak{Ch}/S \to \op{Grp}$ such that $\phi'(\mc X)(\mc U \to \mc Y) = V_\mc X(\Phi' (\mc Y)/\Phi'(\mc U))$. Then we have the following equivalences of functors:
\begin{itemize}
  \item $\phi'((\bb Z \times \op{Gr})_\bb Q) =$ the fibered Picard category over $\mathfrak R \mathfrak{Ch}^c/S$ that is the fundamental groupoid of $K$-cohomology with support.
  \item Let $n$ be an integer. Then $\phi'(\bb P^\infty[\frac{1}{n}]) = \mathfrak{Pic}[\frac{1}{n}]$ and $\phi'(\bb P^\infty_\bb Q) = \mathfrak{Pic}_\bb Q$, the fibered category of line bundles localized at $n$ or $\bb Q$ over
  $\mathfrak R \mathfrak{Ch}/S$, associating to any object of $\mathfrak R \mathfrak{Ch}/S$ the category of localized linebundles thereupon.
\end{itemize}
\end{Prop}
\begin{proof} The first statement is essentially by definition. Consider the second statement. For a simplicial sheaf $\mc X$ and a sheaf of groups
$G$, a $G$-torsor is a simplicial sheaf $\mc Y \to \mc X$ with a free action of $G$ such that $\mc Y/G = \mc X$. In other words, a collection of
$G$-torsors $\mc Y[n]$ on $\mc X[n]$ such that for a morphism $\phi: [n] \to [m]$ there are induced morphisms $\phi^*$ interchanging the data in
the obvious manner. Now, it follows from \cite{VoevMorel}, Section 4, Proposition 3.8 that for a simplicial sheaf $\mc X$, $\phi'(\bb P^\infty)(\mc X)$
is the category of $\bb G_m$-torsors on $\mc X$. Thus, for a regular algebraic space $U$ it is clear that this is the category of line bundles on $U$.
Let $U$ be an regular algebraic stack with smooth presentation $X \to U$ with $X$ an algebraic space. Then $U$ identifies with the simplicial sheaf
whose $n$-simplices are given by $X \times_U X \times_U \ldots \times_U X$, $n$-time, and face and edge-maps by repeated diagonals and projections as face and edge-maps. Since a surjective morphism of line bundles is necessarily injective one verifies that a $\bb G_m$-torsor on $U$ necessarily has isomorphisms as transition-morphisms, and we conclude by smooth descent.
\end{proof}
\begin{remark} By \cite{Vistoli-Kresch}, Lemma 3.2, it follows that a Deligne-Mumford stack $\mc M$, separated and of finite type over a Noetherian base scheme with moduli space $M$,
then $\mathfrak{Pic}(M)_\bb Q \to \mathfrak{Pic(\mc M)}_\bb Q$ is an equivalence of categories.
\end{remark}

A priori the operations given by rigidity are abstract and one might want to relate them to other operations. One standard way of doing so is as follows. Restricting any "virtual" operation given
by a morphism $(\bb Z \times \op{Gr})_\bb Q \to (\bb Z \times \op{Gr})_\bb Q$ along
$\bb P^\infty \to (\{ 0 \} \times \op{Gr})_\bb Q \to (\bb Z \times \op{Gr})_\bb Q$ gives us the behavior of the operation on an actual line bundle where we can often write down explicitly what it does. Then by the splitting principle one can often compare this to other operations.

\subsection{Some consequences of rigidity}
\begin{D} We denote by $F^i W(-)$ (resp. $W^{(i)}$) the category fibered in groupoids $\phi'(\op{Fil}^{(i)})$ (resp. $\phi'(\bb H^{(i)})$) over
$\mathfrak{RCh}^c := \mathfrak{RCh}^c/\spec \bb Z$. We write furthermore $F^i W(\mc U \to \mc X) = F^i W^Z(X)$ and $W(\mc U \to \mc X)^{(i)} = W^\mc Z(X)^{(i)}$. Notice that for an algebraic stack $\mc X$ that is not an algebraic space $F^0 W(\mc X) = W(\mc X)$ is in general not the virtual category $V(\mc X)_\bb Q$ of $\mc X$.
\end{D}
We record the following.
\begin{T} \label{thm:virtualoperations} The functors $F^iW(-)$ have the following properties.
\begin{enumerate}
  \item The functors $F^{i-1} W(-) \to F^{i} W(-)$ are faithful additive functors of Picard categories.
  \item There are pairings, unique up to unique isomorphism, $$F^iW(-) \times F^jW(-) \to F^{i+j}W(-)$$ lifting the pairings $F^i K_0(X)_{\bf Q} \times F^j K_0(X)_{\bf Q} \to F^{i+j}K_0(X)_{\bf Q}$ on regular schemes, such that for $i'\leq i, j' \leq j$, we have a commutative diagram
  $$\xymatrix{F^iW(-) \times F^jW(-) \ar[r] \ar[d] &  F^{i+j}W(-) \ar[d] \\
  F^{i'}W(-) \times F^{j'}W(-) \ar[r] &  F^{i'+j'}W(-) }.$$
  In particular there are pairings
  $$F^i W^{\mc Z}(\mc X) \times F^j W^{\mc Z'} (\mc X) \to F^{i+j} W^{\mc Z \cap \mc Z'}(\mc X).$$
  \item There are unique pairings $W^{(i)} \times W^{(j)} \to W^{(i+j)}$, extending the usual pairings $K_0(X)^{(i)} \times K_0(X)^{(j)} \to K_0(X)^{(i+j)}$ on regular schemes.
  \item The pairings are compatible with the isomorphism $$F^{i} W(-) = F^{i-1} W(-) \times W^{(i)}$$ and they all satisfy the obvious associativity constraints.
  \item The above is compatible with zero-objects in that a zero-object in one variable maps to a zero-object in the second.
  \item The Adams-operations act on all the objects and functors involved, and these operations are moreover, up to unique isomorphism, uniquely defined as liftings of the usual Adams operations.
  \item Let $\mc X$ be a regular algebraic stack of dimension $d$ with finite affine stabilizers and $\mc U \to \mc X$ an open substack such that $\mc X \setminus \mc U$ is of codimension $m$. Then $F^{d+2}W(\mc U \to \mc X)$ is equivalent to the trivial category with one object and the identity as only morphism and we have an equivalence of categories:  $W^\mc Z(\mc X) = \bigoplus_{i = m}^{d+1} W^{\mc Z}(\mc X)^{(i)}$.
\end{enumerate}
\end{T}
\begin{proof} For simplicity we treat the case when $\mc U = \emptyset$, the other cases are analogous. First, (b),(c),(d) and (f) are clear from rigidity. For (a), it is enough to show that for any $\mc X$ and object $x$ of $F^{i-1}W(\mc X)$,
$\op{Aut}_{F^{i-1}W(\mc X)}(x) \to \op{Aut}_{F^{i}W(\mc X)}(x)$ is injective. But this is clear since this map identifies with the injection $F^{i-1} K^{sm}_1(\mc X) \to F^{i} K^{sm}_1(\mc X)$.
Now (e) follows from the description of the pairing in Theorem \ref{thm:BGLfil}. Since we will only be concerned with (g) for a scheme we give the proof in this case. It follows from Corollary \ref{Cor:adams-weight-relation} but we give yet another proof along classical lines when $\mc U = \emptyset$:
\begin{TL} Let $X$ be a regular scheme with $d =
\dim X$, and $i = 0,1$. Recall that $F^j K_i(X)_\bb Q$ is the filtration on $K_i(X)_\bb Q$ determined by \linebreak $F^j K_i(X)_\bb Q = \bigoplus_{p \geq j} K_i(X)^{(p)}$. Then $F^{d+i+1} K_i(X)_\bb Q = 0$.
\end{TL}
\begin{proof} Consider the Quillen coniveau spectral sequence $$E_1^{p,q}(X) = \bigoplus_{x \in X^{(p)}} K_{-p-q}(k(x)) \Longrightarrow K_{-p-q}(X)$$ where $X^{(p)}$ denotes the codimension $p$-points of $X$ and $k(x)$ is the residue field of $x$. By \cite{Souleoperations} Theoreme 4, iv), we have that, for $i =0,1$, $$K_i(X)_{\bb Q} = \bigoplus_{p=0}^d E_2^{p,-p-i}(X)_{\bb Q}.$$ Furthermore, it is remarked in \cite{IntviaAdams}, proof of Theorem 8.2, that the Adams operations $\Psi^k$ act on  the spectral sequence and in particular on $E_r^{p,-p-i}(X)$ by $k^{p+i}$ ($i =0,1$ and $r \geq
1$). Also, $E_r^{p,-p-i}(\Psi^k)$ converge towards the Adams operations on $K_i(X)$. Thus, for $i =0,1$, it follows that $K_i(X)^{(d+i+1+k)} =
0$ for $k \geq 0$ so that $F^{d+i+1} K_i(X)_\bb Q = 0$. \end{proof}
We immediately deduce that the categories $F^i W(X)$ are trivial, \ie all objects are uniquely isomorphic, for $i \geq \dim
X + 2$.
\end{proof}

\begin{remark} The proof of property (g) in the case of a regular algebraic space goes through verbatim. The general case is obtained in a similar way,
but one has to work instead with the spectral sequence $E_1^{p,q}(\mc X) = \bigoplus_{\xi \in X^{(p)}} K^{sm}_{-p-q}(G_{\xi, \op{red}}) \Longrightarrow K^{sm}_{-p-q}(\mc X)$
which exists by a Brown-Gersten argument applied to the flabby $S^1$-spectrum representing cohomological $K$-theory and by virtue of
$G^{sm}_p = K^{sm}_p$ by Poincaré duality for the cohomology of the $K$-theory for regular algebraic stacks (see Theorem \ref{Poincare-duality}). Then each $G_{\xi, \op{red}}$ is a gerbe banded by some reduced algebraic group $H$, which is in fact necessarily an abstract finite group over the algebraic closure of the moduli space. To understand the Adams operations we can by étale descent moreover suppose that the moduli space $\spec k(x)$ of $G_{\xi, \op{red}}$ is separably closed so that the gerbe is trivial and $G_{\xi, \op{red}} = [\spec k(x)/H]$. By the arguments of \cite{Thomason7}, 2.3 there is a spectral sequence
 $$E_1^{p,q} = K_q(\prod^p H) \to K_{q-p}^{sm}(G_{\xi, \op{red}}).$$
 Then $K_i^{sm}(G_{\xi, \op{red}})_\bb Q = K_i(\spec k(x))_\bb Q^H$ for all $i$. The Adams operations $\Psi^k$ act on $K_i^{sm}(G_{\xi, \op{red}})$ via the restriction of $K_i(\spec k(x))$ to the $H$-invariant part and thus by $k^i$. The rest is similiar but skipped.
\end{remark}
\begin{remark} From \cite{Levine-motivic}, Theorem 11.5 it follows that if $R$ is a Dedekind domain, and $X$ is a regular finite type $\spec R$-scheme, then the $\gamma$-filtration on $K_n(X)$ for any integer $n$ terminates after $d+n+1$ steps.
\end{remark}

We harvest some obvious corollaries:

\begin{Cor} \label{Cor:adams-rigid-compatibility} Let $\mc X$ be a regular algebraic stack with closed substack $\mc Z$. The Adams operations on $W^\mc Z(\mc X)$ are compatible with the Adams operations constructed on $V^\mc Z(\mc X)$ in Proposition \ref{Prop:Adamsoperation}
under the functor $V^\mc Z(\mc X) \to W^\mc Z(\mc X)$. Moreover, there is a determinant functor $\det: W(\mc X) \to \mathfrak{Pic}(\mc X)_\bb Q$ such that the diagram
$$\xymatrix{V(\mc X) \ar[r] \ar[d]^{\Psi^k} & W(\mc X) \ar[d]^{\Psi^k}\\
V(\mc X) \ar[r] \ar[d]^{\det} & W(\mc X) \ar[d]^{\det} \\
\mathfrak{Pic}(\mc X) \ar[r] & \mathfrak{Pic}(\mc X)_\bb Q}$$
commutes up to canonical natural transformation.
\end{Cor}
\begin{proof} By rigidity the two Adams-operations coincide on line bundles and we conclude by the splitting principle. Moreover, $\mathfrak{Pic}(\mc X)_\bb Q$ clearly satisfies coherent descent since it is a localization of the category $\mathfrak{Pic}(\mc X)$ which does. The determinant functor then exists by cohomological descent and the diagram commutes again by rigidity and the splitting principle.
\end{proof}

\begin{Cor} \label{Cor:lambda} There are unique $\lambda$-operations on $F^iW(-)$
$\lambda$-operations
satisfying, for a regular algebraic stack $\mc X$,
$$\lambda^k(x+y) = \sum_{j= 0}^k \lambda^j(x) \otimes \lambda^{k-j}(y)$$
and moreover for a bounded complex of vector bundles $E_\bullet$ one has $t_{E_\bullet}: \lambda^k E_\bullet \simeq \wedge^k E_\bullet$, where $\wedge$ refers to the already constructed operations in Proposition \ref{prop:lambdasoule}. For an exact sequence of bounded complexes of vector bundles
$0 \to E'_\bullet \to E_\bullet \to E''_\bullet \to 0$ a commutative diagram of isomorphisms
$$\xymatrix{\lambda^k(E_\bullet) \ar@{=}[d]^t & = & \sum_{j = 0}^k \lambda^j(E'_\bullet) \otimes \lambda^{k-j}(E''_\bullet) \ar@{=}[d]^t  \\
\wedge^k(E_\bullet) & = & \sum_{j = 0}^k \wedge^j(E'_\bullet) \otimes \wedge^{k-j}(E''_\bullet) }$$
with the lower row defined as in (\ref{lambda-2}).
\end{Cor}
\begin{proof} \mps{I think to clarify the argument in this corollary would just involve inserting an obvious diagram, I think it might be ok to leave it as it is. I added explanation as to why I want $Rp_* 1 = 1$.} We assume again $\mc U = \emptyset$, the general result follows from an argument of the type in Proposition \ref{Prop:Iversen}. We might also suppose that the complexes of vector bundles are actually vector bundles, by the equivalence of the virtual category of complexes of vector bundles and vector bundles and the equivalence of the constructed $\wedge$-operations (cf. Proposition \ref{prop:virtualvectorcomplexes} and Proposition \ref{prop:lambdasoule} (c)). \\ Unicity is then clear by the splitting principle. Existence of the $\lambda$-operations are given by rigidity, and we suppose for simplicity that $i = 0$. For a line bundle rigidity provides us with an isomorphism $\lambda^k L \simeq \wedge^k L = L$ if $k = 1$ and an isomorphism with $0$ if $k > 1$. The existence of $t$ now follows from Corollary \ref{Cor: adams-lambda-relation} and Corollary \ref{Cor:adams-rigid-compatibility}.
To prove the additive property, it is clear given a complete flag of $E$ compatible with $E'$. We need to prove that it is independent of the choice of flag and one proceeds by induction on the rank of $E$. Suppose $L$ is a locally split sublinebundle of $E$ of rank 2. This defines $t_{E, L}: \lambda^k E \simeq \sum_{j=0}^2 \lambda^j L \otimes \lambda^{k-j} E/L$ compatible with $\wedge$ as in the theorem by additivity of the Adams operations and Corollary \ref{Cor: adams-lambda-relation} and the explicit description of the Adams operations in \ref{Prop:Adamsoperation}. In general, given a filtration $F' \subset E$ one defines an isomorphism $t_{E,F'}: \lambda^k E \to \lambda^j F' \otimes \lambda^{k-j} E/F'$ by requiring the diagram in the statement to commute. We need to show that this is the operation given by rigidity. Given two filtrations $F'' \subset F' \subset E$ one verifies that by induction hypothesis that $t_{E, F''} = t_{E, F'}$ in a way stable by basechange of regular schemes. Given two unequal line bundles $L' \subset E$ and $L'' \subset E$ we consider the Grassmannian $p: \op{Gr}_{L',L'',2}(E) \to X$ classifying bundles $M$ of rank 2 such that $L', L'' \subset M \subset E$. Then $Rp_* 1 = 1$ so that by the projection formula $A = Rp_* (Lp^* A)$ for any $A$ and it follows that $t_{p^* E, p^* L'}$ determines $t_{E, L'}$. But by what we established $t_{p^* E, p^* L'} = t_{p^* E, \mc M} = t_{p^* E, p^* L''}$ for $\mc M \subset p^* E$ the universal rank 2 sub bundle on $\op{Gr}_{L,L',2}(E)$ and we conclude.
\end{proof}

\begin{Cor} \label{Cor:gamma} Let $\mc X$ be as in (g) in Theorem \ref{thm:virtualoperations}. There are $\gamma$-operations on the virtual category $W(-)$, $j \geq 2$, $\gamma^j$, inducing the natural operations on $K_{0,\bb Q}$.
For a virtual bundle of rank 0, $$\gamma^j(v) \in F^j W.$$
Furthermore, for any two virtual objects $x$ of rank $0$, we have a family of isomorphisms in $F^{k}W$, functorial in $x$ and $y$;
\begin{equation} \label{eq:gamma-additivity} \gamma^k(x + y) \simeq \sum \gamma^j(x) \otimes \gamma^{k-j}(y).
\end{equation}
Since for a line bundle $L$, $1-L$ identifies with an object of $F^1 W$, we have that $(1-L)^i$ is an object of $F^i W$. We then have canonical isomorphisms in $F^iW$:
$$\gamma^i(1-L) \simeq (1-L)^i$$
and $$\gamma^i(L-1) = 0 \hbox{ for } i \geq 2.$$ Given a vector bundle $E$ of rank $n$, we have a canonical isomorphism $\gamma^n(E_\bullet-n) = (-1)^n \lambda_{-1}(E):= (-1)^n \sum_{j=0} (-1)^j \lambda^j E$ and for any $k > 0$ a trivialization $\gamma^{k+n}(E-n) = 0$ such that for a short exact sequence of vector bundles $0 \to E' \to E \to E'' \to 0$ of ranks $n', n$ and $n''$ we have an isomorphism $$\xymatrix{\gamma^n(E-n) \ar@{=}[d]  & = & \sum_{i=0}^n \gamma^i(E'-n') \otimes \gamma^{n-i}(E''-n'') \ar@{=}[d] \\
(-1)^{n} \lambda_{-1}(E) & = & (-1)^{n'} \lambda_{-1}(E') \otimes (-1)^{n''}\lambda_{-1}(E'') }.$$ Thus the functor $E \mapsto \lambda_{-1} E$ defined on vector bundles has essential image in $F^{\op{rk} E}W$.  Moreover, for a vector bundle $E$, the trivialization $\gamma^{k+n}(E-n) = 0$ is compatible with the trivializations given by (\ref{eq:gamma-additivity}) and admissible filtrations.
\end{Cor}

\begin{proof} All statement except the last one are direct consequences of rigidity. By definition (cf. \cite{Riemann-RochAlgebra}, chapter III) and rigidity we are given a relationship of $\gamma$ and $\lambda$-operations
$\sum_{i=0} \gamma^i(u) t^i = \sum_i \lambda^i(u) \left( \frac{t}{1-t} \right)^i$
so that, in view of that $1/(1-t)^{r+1} = \sum {j+r \choose j} t^j$, the relationship, for $k > 0$, in $W(\mc X)$
$$\gamma^k(u) = \sum_{i=0}^k \lambda^i(u) {k-1 \choose k-i} = \lambda^k(u + k - 1).$$ If $u$ is a virtual bundle of rank 0 we deduce the equality in $F^kW$ compatible with sums and the product on the filtration.
 Now, for a line bundle $L$ one has $\lambda^n(-L) = (-L)^n$ and thus one has for a vector bundle of rank $n$
 a canonical isomorphism $\gamma^n(E-n) = \lambda^n(E-1) = (-1)^n \lambda_{-1}^n (E)= (-1)^n \lambda_{-1}(E)$, where the latter isomorphism is an isomorphism in $V(\mc X)$.
 We hereby identify $\lambda_{-1}(E)$ as an element in $F^n W(\mc X)$. One also obtains a trivialization $\gamma^{k+n}(E-n) = \lambda^{k+n}(E + k -1) = \wedge^{k+n}(E+k-1) = 0$ in $W(\mc X)$. We need to verify that the isomorphism lies in $F^{k+n}W(\mc X)$. We proceed by induction. For a line bundle this is given by rigidity. Suppose we have verified the all the given statements for vector bundles of rank strictly less than $n$ for all regular schemes. Suppose $E$ is of rank $n$. Given an admissible filtration $F'' \subset F' \subset E$ of with $F''$ and $F'$ of ranks  $n''$ and $n'$ respectively consider the following diagram

 $$\scalebox{0.8}[1]{\xymatrix{\sum \gamma^j(F'-n') \gamma^{n-j}(E/F' - (n-n'')) \ar[d] \ar[r] & \sum \gamma^i(F'' - n'') \gamma^{j}(F'/F'' - (n' - n'')) \gamma^{n-i-j}(E/F' - (n-n')) \ar[d]\\
 (-1)^{n'} \lambda_{-1} (F') (-1)^{n-n'}\lambda_{-1}(E/F') \ar@{=}[r] \ar@{=}[d] & (-1)^{n'} \lambda_{-1}(F'') (-1)^{n'-n''}\lambda_{-1}(F'/F'') (-1)^{n-n'}\lambda_{-1}(E/F') \ar@{=}[d]  \\
 (-1)^n \lambda_{-1} (E) \ar@{=}[r] & (-1)^{n''} \lambda_{-1}(F'') (-1)^{n-n''}\lambda_{-1}(E/F'')  \\
 \gamma^n(E-n) \ar[u] \ar[r] \ar@/^5pc/[uuu] & \sum \gamma^i(F'' - n'') \gamma^{n-i}(E/F'' - (n - n'')) \ar[u] \ar@/_8pc/[uuu] }}$$
 where the morphisms are given by the various trivializations $\gamma^i(\diamondsuit - \op{rk} \diamondsuit) = 0$ in $F^iW$ whenever $i > \op{rk} \diamondsuit$ and the isomorphism
 $\gamma^{\op{rk} \diamondsuit}(\diamondsuit - \op{rk} \diamondsuit) = \lambda_{-1}(\diamondsuit)$ given by
 induction hypothesis. The outer contour commutes by definition of the $\gamma$-operations. The middle square commutes by compatibility with admissible filtrations of (\ref{lambda-1}), the upper
square commutes by induction hypothesis as does the right hand diagram. Thus the diagram determined by $F' \subset E$ commutes if and only if the diagram determined by
the diagram determined by $F'' \subset E$ commutes. Arguing as in the previous corollary one sees that this morphism is independent of choice
of $F''$ and $F'$ and by the splitting principle one obtains that the diagram associated to $F' \subset E$ commutes. In the same way we obtain a canonical trivialization
$\gamma^{n+k}(E - n) = 0$ in $F^{n+k}W(\mc X)$ for $k > 0$. We need to verify that its image in $W(\mc X)$ coincides with the trivialization
$t: \lambda^{k+n} (E + k - 1) = \wedge^{k+n}(E + k -1) = 0$. This follows by additivity and induction on $k$.

\end{proof}
\begin{Cor} \label{cor:nilpotentfiltration} Let $X$ be a regular scheme of dimension $d$.
Then for any virtual bundle $v$ in $W(X)$ of rank 0, $k > 1$, $$\gamma^{d+k}(v) \simeq 0$$
canonically.
\end{Cor}
\begin{proof} This follows from Theorem \ref{thm:virtualoperations} and Corollary \ref{Cor:gamma}.
\end{proof}

\begin{Prop} \label{Prop:adamsdecomposition} Let $X$ be a regular scheme of
dimension $d$ and $Z$ a closed subscheme of $X$. Then for any $k$, and a virtual bundle $v$, an isomorphism $\Psi^k(v) = k^i v$ in $W^Z(X)$ defines a projection of $v$ into $W^{Z, (k)}$. Thus it implies that there is a canonical isomorphism $\Psi^n (v) = k^n v$ for any $n$.
\end{Prop}
\begin{proof} By Theorem \ref{thm:virtualoperations} there is an equivalence $W^Z(X) = \bigoplus_{j=0}^{d+1} W^{Z}(X)^{(j)}$ so that $v$ is
equivalent to an object of the form $\sum \pi_i v_i$ with $\pi_i: W^{(i)} \to W(X)$. Applying $\Psi^j$ we obtain
an isomorphism $(k^j - k^i)v_i = 0$ in $W^{Z}(X)^{(i)}$ and so $v_j = 0$ for $j \neq 0$.
\end{proof}

The following is trivial but gives an idea of what the second step of the filtration looks like:

\begin{Cor} Let $X$ be a regular scheme. Any virtual bundle of rank 0 and trivialized determinant bundle defines an object of $F^2 W(X)$.
\end{Cor}

\begin{proof} The functor $R: W(X) \to W(X)$ associating to a virtual bundle $u$ the virtual bundle
$$R(u) = (v - \op{rk} u) - (\det u - 1)$$
has essential image in $F^2W$ by rigidity. For $u$ and $v$ both virtual bundles on $W(X)$ there is a canonical isomorphism
$$R(u+v) = R(u) + R(v) - (\det u - 1)(\det(v - 1)$$
in $F^2W(X)$ where the product $(\det u -1)(\det v -1)$. These are stable by pullback of regular schemes. They moreover correspond to the isomorphisms defined by
$\op{rk}(u+v) = \op{rk} u + \op{rk} v$ and $$\det(u+v)-1 = (\det u - 1)(\det v - 1) + (\det u - 1) + (\det v - 1)$$ defined as in \cite{determinant}, (9.7.8). Thus, a virtual bundle of rank 0 and with trivialized determinant bundle defines an object of $F^2W(X)$ and addition of two such data defines addition in $F^2W(X)$. \end{proof}


\section{Comparison with constructions via algebraic cycles}

In \cite{ChowCT} and \cite{AAG} there is a careful outline of his notion of Chow categories and Chern functors. This section provides a comparison of these constructions to the above ones, at least in some special cases.
\subsection{Rational Chow categories}
\begin{spacing}{1}
Let $Z$ be a scheme which we suppose for simplicity is connected, the general construction is analogous. Suppose furthermore that $Z$ is separated and of finite type over a non-fixed regular scheme $S$ and admits an ample family of line bundles. By supposing that $Z$ has an ample family of line bundles, by Jouanolou-Thomason (cf. \cite{Weibel}, Proposition 4.4) $Z$ has the $\bb A^1$-homotopy type of an affine scheme and by Corollary \ref{Cor:homotopyinvariance} below we can suppose that $Z$ is affine.

\begin{D} Define the (rational) cohomological Chow category with support, denoted $\mc{CH}^i_Z(X)$, as the category $W^Z(X)^{(i)}$ considered in the last section.
\end{D}
Recall also the construction of Chow categories of J. Franke. For a scheme $X$, write $X^{(p)}$ for the codimension $p$-points of $X$. Suppose $X$ is regular and connected and that $Z$ is a closed subscheme of $X$. Consider the coniveau spectral sequence
$$E_1^{p,q} = \bigoplus_{x \in X^{p} \cap Z} K_{-p-q}(k(x)) \Longrightarrow K_{-p-q}^Z(X)$$
giving rise to
$$\bigoplus_{z \in X^{(p-1)} \cap Z} K_{1-p-q}(k(z)) \stackrel{d_1}\longrightarrow \bigoplus_{y \in X^{(p)} \cap Z} K_{-p-q}(k(y)) \stackrel{d_0}\longrightarrow \bigoplus_{x \in X^{(p+1)} \cap Z} K_{-p-q-1}(k(x)).$$
Notice that since $X$ is catenary, $Z \cap X^{(p)} = Z_{(\dim X - p)}$, where $Z_{(d)}$ denotes the dimension $d$-points of $Z$. Without the assumption that $X$ is regular and taking $Z$ empty, this is the classical coniveau spectral sequence converging to $G$-theory. We have the following from \cite{ChowCT}:
\begin{D}
The category $\hbox{\textbf{CH}}^i (Z)$ is defined as follows. The objects are given by codimension $i$-cycles on $Z$, and homomorphisms between two cycles $z$ and $z'$ are described by
 $$\Hom(z,z') := \{f \in E_1^{i-1, -i}(X), d_1(f) = z' - z  \}/d_1 E_1^{i-2, -i}(X). $$
 The category $\hbox{\textbf{CH}}^i (Z)_\bb Q$ is then the category $\hbox{\textbf{CH}}^i (Z)$ localized at $\bb Q$. They both have natural structures of Picard categories.
\end{D}
\begin{D} Let $X$ be regular and $Z$ be a closed subscheme of $Z$. The $\dim X - i$-th (homological) Chow category is the following: The objects are elements of $Z_{\dim X - i}(Z) = Z^i_Z(X)$ and whose homomorphisms are given as $\Hom(z,z') := \{f \in E_{1,Z}^{i-1, -i}(X), d_1(f) = z' - z  \}/d_1 E_{1,Z}^{i-2, -i}(X)$. Here $Z^i_Z(X)$ denotes the cycles of codimension $i$ on $X$ with support on $Z$. It is moreover clear from the description of the coniveau spectral sequence that the categories are independent of $X$, since it reduces to the niveau spectral sequence for $Z$ and $X^{(i)} \cap Z = Z_{(\dim X - i)}$ since $X$ is catenary. The category localized at $\bb Q$ is the rational (homological) Chow category and we denote it by $\mc {CH}_{\dim X - i}(Z)$. In view of that for a field $k$, $K_2(k) = k^* \otimes_\bb Z k^*$ modulo symbols of the form $(x, 1-x), x \in k \setminus \{0,1\}, K_1(k) = k^*$ and $K_0(k) = \bb Z$, it is thus a Picard category related to a complex
 $$\bigoplus_{z \in X^{(i-2)} \cap Z} k(z)^* \otimes_\bb Z k(z)^* \stackrel{d_1}\longrightarrow \bigoplus_{y \in X^{(i-1)} \cap Z} k(y)^* \stackrel{d_0}\longrightarrow \bigoplus_{x \in X^{(i)} \cap Z} \bb Z$$
defined elementary in terms of fields and where $d_0$ is the associated divisor to a rational function and $d_1$ is the tame symbol.
 \end{D}
We included $X$ to emphasize that it is clearly equivalent to the category $\hbox{\textbf{CH}}^i (Z)_\bb Q$, as noted in the next proposition, but it has an obvious definition in terms of the niveau filtration.
\begin{Prop} [Poincaré duality] \label{prop:ChowCatequiv} With these notations there is we have the following.

\begin{enumerate}
  \item The category $\mc {CH}^i_Z(X)$ be identified with the cycle-groupoid giving a Poincaré duality-type equivalence:
  $$\mc {CH}_{Z}^{i}(X) \simeq \mc {CH}_{\dim X - i}(Z).$$
  \item Suppose moreover $Z$ is regular, then there is a natural equivalence of Picard categories
  $$\Psi: \hbox{\textbf{CH}}^i (Z)_\bb Q \to \mc{CH}_{\dim X - i}(Z)$$.
\end{enumerate}

compatible with pushforward.
\end{Prop}
\begin{proof} We already mentioned the second part is true. For the first, given a cycle $z$ in the latter category, associate to it the edge in $G$-theory. $E_2^{k-1,-k}(X)$ in turn admits a canonical map to the $G$-theory space by sending an element $f \in k(x)^*, x \in X^{(k-1)} \cap Z$ to the associated map of linebundles $\calo \simeq \calo(\op{div}  f)$ on $\overline{\{x\}}$. By \cite{IntviaAdams} the projection onto the Adams eigenspace this induces an isomorphism on $\pi_0$. The statement for $\pi_1$ is an analogous result which is not made explicit therein but which follows analogously. Thus the map induces an equivalence of categories. \\
\end{proof}
\end{spacing}

For a flat morphism there is the naive pullback $Lf^* : G(X) \to G(Y)$ on $G$-theory spaces and hence $Lf^*: C(X) \to C(Y)$ on the associated virtual categories. Composing it with the suitable projections and inclusions we obtain:
\begin{Prop} [Flat pullback] Suppose there is a flat morphism of schemes $f: X \to Y$ of relative dimension $d$. Then there is a natural functor $f^*: \mc{CH}_i(Y) \to \mc {CH}_{i+d}(X)$ compatible with composition. It is moreover compatible with the induced map on supported virtual categories in case of a compatible diagram. It maps an $i$-dimensional cycle $V$ to the closure of $f^{-1}V$ in $X$.
\end{Prop}
The proof is obvious from niveau spectral sequence and the description of the map $Lf^*$ which is part of the definition of pullback of cycles.
\begin{remark} \label{remark:fiberedPicard-flat} In the language of \cite{ChowCT}, section 3.6, this makes the association $Z \mapsto \mc {CH}^i(Z)= \mc {CH}_{\dim Z - i}(Z)$ with flat pullbacks into a fibered Picard category over the schemes for which these are defined, and moreover satisfy Poincaré duality. In other words, a contravariant assignment of Picard categories $Z \mapsto P_Z$ for appropriate $Z$ together with associativity and composition constraints for the class of pullback morphisms.\end{remark}
\begin{Cor} [Homotopy invariance] \label{Cor:homotopyinvariance} Let $Y \to Z$ be a torsor under a vector bundle on $Z$, and $Z$ of finite type over a regular scheme $S$. Then the flat pullback $f^*: \mc{CH}_i(Z) \to \mc {CH}_{i+d}(Y)$ induces an equivalence of categories.
\end{Cor}
\begin{proof} We need to show it induces an isomorphism on $\pi_0$ and $\pi_1$. These group however sit in a long exact homotopy sequence for localization. Thus if the induced map on all homotopygroups of the spaces giving Adams eigenspaces are isomorphisms we can localize. For this, by Mayer-Vietoris, we suppose that $Z$ is affine. In this case the torsor is trivial and equal to a vector bundle and localizing further we can suppose it is $\bb A^n_Z \to Z$. In this case the map fits into a square, for a regular $X$, $$\xymatrix{\bb A^n_Z \ar[r] \ar[d] & \bb A^n_X \ar[d] \\ Z \ar[r] & X}$$  and the induced map on $Lf^*: K_i^Z(X) \to K_i^{\bb A^n_Z}(\bb A^n_X)$ coincides with $Lf^*: G_i(Z) \to G_i(\bb A^n_Z)$ where there is homotopy invariance. The corollary follows after passing to rational coefficients.
\end{proof}
\begin{Prop} [Gysin-type morphism] \label{Prop:pullbackchow} Suppose that $Z \to X$ and $Z' \to X'$ are closed embeddings of schemes into regular schemes $X$ and $X'$ and that there is a commutative square (which we call a compatible diagram)
$$\xymatrix{Z \ar[r] \ar[d]^f & X \ar[d]^F\\
Z' \ar[r] & X'}$$
with $f$ and $F$ any two morphisms of schemes. Then there is a functor $$f^*: \mc {CH}^i_{Z'}(X') \to \mc {CH}^i_{Z}(X)$$ compatible with composition of these types of diagrams. If $f$ is a projective \lci morphism of relative dimension $d$, then this also induces a functor $f^!: \mc {CH}_i (Z) \to \mc {CH}_{i+d} (Z')$ independent of $X$ and $X'$.
\end{Prop}

\begin{proof} Given the above data there is clearly a functor $Lf^*: \bb H_{Z'}(X')^{(i)} \to \bb H_{Z}(X)^{(i)}$ which induces the map in the first part of the proposition.  For the construction of the other map, one uses the flat map to define the pullback along projective bundle projections. Standard techniques reduces us to consider the case of a regular closed immersion $f$. Suppose first that $Z'$ is included in a regular scheme $X$. Then the diagram
$$\xymatrix{\mc {CH}^i_{Z'}(X) \ar[d]^{LF^*} \ar[r] & V^{Z'}(X) \ar[d]^{LF^*} \ar[r] & C(Z') \ar[d]^{Lf^*} \\
\mc {CH}_{Z}^i(X) \ar[r] & V^{Z} (X) \ar[r] & C(Z) }$$
commutes up to equivalence, where the morphisms are the obvious ones and the existence of $Lf^*$ follows from the \lci assumption. This shows in particular that the induced functor $Lf^*: \mc {CH}_i (Z') \to C(Z)$ has essential image in $\mc {CH}_{i-d} (Z)$ and the functor $Lf^*$ is independent of $F$. Whenever $Z'$ does not admit a closed immersion into a regular scheme $X$, we can find by Jouanolou-Thomason (cf. \cite{Weibel}, Proposition 4.4) an affine torsor under a vector bundle on $Z'$ an pullback along this torsor induces, by Corollary \ref{Cor:homotopyinvariance}.

assume that $X$ is moreover affine
This provides the definition for $f^!$. \mps{This proof is maybe not nonsense anymore, fix so that independent of $X$ and $X'$}
\end{proof}

The following follows from the correspond result in $G$-theory.
\begin{Prop} [Topological invariance] Let $Z$ be a closed subscheme of a regular scheme $X$. If $Z_{red}$ denotes the associated reduced scheme the natural map gives an equivalence of categories
$$\mc {CH}_i(Z) = \mc {CH}_i(Z_{red}).$$
\end{Prop}

\begin{Prop} [Proper pushforward]\label{prop:pushforwardchow} Suppose $f: Z \to Z'$ is a proper morphism of schemes. There is an induced pushforward $f_* : \mc {CH}_k (Z) \to \mc {CH}_k (Z')$.
\end{Prop}

\begin{proof} The proof of \cite{Gillet-Riemann-Roch}, Theorem 7.22 (iii) shows that there is a push-forward on the complexes given by the first page of the niveau spectral sequence. It follows that there is a pushforward $f_*: \mc {CH}_i(X) \to \mc {CH}_i(Y)$ on the cyclelevel, already without rational coefficients, given by the classical formulas on the level of objects (cf. \cite{intersection}, section 1.4).
\end{proof}
\begin{D} \label{defn:filbydimension} Define $\op{Fil}_k(Z) = \bigsqcup_{i \leq k} \mc {CH}_i(Z)$. By the proof of Proposition \ref{prop:ChowCatequiv}, if $X$ is a regular containing $Z$, this is equivalent to $F^{\dim X - k}_Z(X)$ considered in Theorem \ref{thm:virtualoperations}. Moreover, we have $$\op{Fil}_k(Z)/\op{Fil}_{k-1}(Z) \simeq \mc {CH}_k(Z)$$ where the quotient denotes cokernel of an additive functor of Picard categories (cf. Definition \ref{defn:determinantfunctorimage}).
\end{D}
Notice that if $f: Z \to Z'$ is a proper morphism of schemes, the functor $Rf_*: C(Z) \to C(Z')$ restricted to $\op{Fil}_k(Z)$ has essential image in $\op{Fil}_k(Z')$. Indeed, we know from the above that any object in $\op{Fil}_k(Z)$ is represented by sheaves of the form $\calo_V$ for closed subschemes $V$ of $Z$ of dimension at most $k$. It is obvious that $Rf_* \calo_V$ is a sum of sheaves with support on sheaves with support on points with dimension at most $k$ on $Z'$. We need to prove that the induced map on $Rf_*:G_1(Z)_\bb Q \to G_1(Y)_\bb Q$ has the same properties with respect to the niveau filtration. Now $Z'$ is defined over a regular scheme $S$, and by Jouanolou-Thomason we can assume first that $S$ and then that $Z'$ is affine. By Chow's lemma we can find a morphism $p: \widetilde{Z} \to Z$ such $p$ and $f p$ are projective and that $Rp_* (\calo_{\widetilde Z}) = \calo_Z$ and hence we can replace, if we can show that $Rp_*:G_1(\widetilde Z)_\bb Q \to G_1(Z)_\bb Q$ is surjective on every step of the niveau filtration, $Z$ by $\widetilde Z$ and thus that $f$ itself is projective. To prove this we can also replace $Z$ by an affine scheme so there is a closed embedding $Z \subseteq \bb A^n_S$ for suitable $n$ and since $p$ is projective it fits into a natural compatible diagram (cf. Proposition \ref{Prop:pullbackchow}). Then the statement follows from the already cited \cite{Gillet-Riemann-Roch}, Theorem 7.22 (iii) which gives the map on the first page of the niveau spectral sequence where surjectivity is clear, and \cite{Souleoperations}, Théorème 4, iv) from which it follows that the isomorphic coniveau spectral sequence with supports in question degenerates at the second page when passing to rational coefficients. The same argument also treats the case of general projective morphisms. With this pushforward we have the following corollary:
\begin{Cor} Let $f: Z \to Y$ be a proper morphism of schemes. The pushforward $Rf_*: C(X)_\bb Q \to C(Y)_\bb Q$ restricted to $\op{Fil}_k X$ has essential image in $\op{Fil}_k (Y).$ The induced functor $f_*: \op{Fil}_k(Z)/\op{Fil}_{k-1}(Z) \to \op{Fil}_k(Y)/\op{Fil}_{k-1}(Y)$ is equivalent to the pushforward on Chow categories.
\end{Cor}
\begin{proof} For the second one, we note that the comparison is trivial in the case of close immersions and to compare the two pushforwards $f_* (\calo_V)$ and $Rf_* (\calo_V)$, we can then assume $V = Z$ and $Y = f(Z)$. If $f: Z \to Y$ is not generically finite, $f_* \calo_Z = 0$ and $Rf_* \calo_Z$ is in $\op{Fil}_{k-1}(Y)$. Generally, pick an open $U$ of $Y$ such that $f$ is finite of degree $d$ and $\calo_Z$ free over $\calo_Y$. In this case, the pushforward is described as $\calo_{U}^d$. In general, by \cite{Riemann-RochAlgebra}, Ch. IV, Lemma 3.7 there is a coherent sheaf $\mc G$ on $Y$, surjective morphisms $\mc G \to \calo_Y^d$ and $\mc G \to R^0 f_* \calo_Z$ such that the restriction to $U$ is an isomorphism and is compatible with the isomorphism $Rf_* \calo_Z|_U = R^0 f_* \calo_Z|_U \simeq \calo_U^d$. This defines an isomorphism of the two pushforwards, and it doesn't depend on $\mc G$. Indeed, we only need to prove that if we have a surjective morphism of coherent sheaves $g: \mc G' \to \mc G$ on a scheme $Z$ which is an isomorphism over some dense open $U$ whose complement is of dimension $k-1$ or less, then it defines an isomorphism in $\mc {CH}_k(Z)$. In this case $\op{Fil}_k(Z) = V(Z)_\bb Q$ and $\mc {CH}_k(Z)$ don't have any automorphisms and is equivalent to the group of zero-cycles of dimension $k$ on $Z$, so the natural determinant functor $V(Z) \to \op{Fil}_k(Z) \to \mc {CH}_k(Z)$ is realized by taking lengths of coherent sheaves along subschemes of dimension $k$ and any exact sequence defines an isomorphism of an object of dimension less than $k$ with zero.
\end{proof}

\begin{Prop} \label{prop:projectionformulafiltration}Let $f: Z \to Z'$ be a proper morphism. Suppose that $V$ is a virtual vector bundle of $Z$ \mps{whose class $[V]$ in $K_0(Z)_\bb Q$ lies in $F^i K_0(Z)$ where $F^i K_0(Z)$ denotes the Adams filtration of $K_0(Z)_\bb Q$. } Then there is a projection formula isomorphism $$V \otimes Rf_* (\mc F) \simeq Rf_*(Lf^* V \otimes \mc F)$$ in $\op{Fil}_{k}(Z')$ where $\mc F$ is an element of $\op{Fil}_k(Z)$. It is also stable under composition of proper morphisms in the naive way and if furthermore $V$ is of rank 0, this is an isomorphism in $\op{Fil}_{k-1}(Z')$.
\end{Prop}
\begin{proof}
We prove the formula by noting it is true if we replace the filtration by the virtual category of coherent sheaves and then proving that $f_*(f^* V \otimes \mc F) \simeq V \otimes \mc F$ in $\mc {CH}_k$. The formula then follows from the above since the pushforward and pullbacks in question are compatible with the ones on the level of virtual filtration by dimension by the previous corollary and functoriality of the kernel of an additive morphism. By functoriality of pushforward we can suppose that $Z$ is a $k$-dimensional integral scheme, that $A$ is $\calo_Z$ and $f$ is finite of degree $d$ and by the splitting principle (or a slight modification thereof since the projective bundle formula for Chow groups is slightly different from that of $K$-theory in terms of indices) we can assume that $V$ is a line bundle. In this case the formula reduces to show that for a rational section $s$ of $L$, there is an isomorphism $f_* \op{div} f^* s = f_* f^* (\op{div} s) = d(\op{div} s)$ and the projection formula on the level of $G_1$. These are actual equalities that already holds on the level of cycles (cf. Example 1.7.4 of \cite{intersection}) so that the "isomorphism" is actuality the identity-map. For the second part, we can suppose that $V$ is an actual difference $E - e$ where $E$ is a vector bundle and $e = \op{rk} E$. By the splitting principle we can furthermore suppose $E$ is a sum of line bundles so we can suppose $E = L$ is a line bundle. By the previous argument, it is enough to verify that the two sides are naturally equivalent to the identity equivalence of the 0-functor on $\mc {CH}_k(Z') = \op{Fil}_k(Z')/\op{Fil}_{k-1}(Z')$ and to prove this we start by proving that for any line bundle $L$ of $Z$, the functor $(L-1) \cap - : \op{Fil}_k(Z) \to C(Z)$ has essential image in $\op{Fil}_{k-1}(Z)$. Any object in the source category $\op{Fil}_k (Z)$ can be written as a sum of objects $\sum i^y_* \mc F$ where $i^y_*$ is the pushforward associated to a point $y$ in $Z$ such that $Y :=  \overline {z}$ has $\dim Y \leq k$ and $\mc F$ is a coherent sheaf (with rational coefficients). By already established projection formula $(L-1) \otimes i^z_* (\mc F) \simeq i^z_* ({i^z}^* (L-1) \otimes \mc F)$ and so we see that we must show that ${i^z}^* (L-1) \otimes \mc F$ has essential image in $\op {Fil}_{k-1}(Z)$. We can thus assume $Z = Y$ and $\mc F = \calo_{Z}$. Given a rational section $s$ of $L$ the divisor defines a sheaf whose support is of dimension $k-1$ and thus an element of $\op {Fil}_{k-1}(Z)$. Another choice of rational section defines an isomorphism of the two sheaves and it is obvious the data glues together to an object ${i^y}^* (L-1) \otimes \mc F$ in $\op {Fil}_{k-1}(Z)$. Moreover, by the arguments of the proof of \cite{Filtrations-Gillet-Soule}, Theorem 4, ii), there natural pairing of $K-$ and $G$-theory restricts to a pairing $F^1 K_0(X) \times F^{\op{dim}}_k G_1(X) \to F^{\op{dim}}_{k-1} G_1(X)$ where $F^1 K_0(X)$ denotes the virtual bundles of $K_0(X)$ of rank 0 and $F^{\op{dim}}_{k-1} G_1(X)$ denotes the filtration by dimension of $G_1(X)$. Thus the automorphisms of $F^{\op{dim}}_{k} G_1(X)$ map into the automorphisms of $F^{\op{dim}}_{k-1} G_1(X)$. By passing to quotients we obtain the result.
\end{proof}
\begin{remark} Since the diagram $$\xymatrix{V' \otimes V \otimes Rf_* (\mc F) \ar[r] \ar[d] & V' \otimes Rf_*(Lf^* V \otimes \mc F) \ar[d] \\
Rf_*(Lf^* (V' \otimes V) \otimes \mc F)\ar[r] & Rf_* (Lf^* V' \otimes Lf^* V \otimes \mc F) }$$
is commutative any virtual bundle of the form $(E_1 - e_1)(E_2 - e_2) \ldots (E_m - e_m)$ where each $E_i$ is a vector bundle of rank $e_i$ defines by multiplication a functor $\op{Fil}_k(Z) \to \op{Fil}_{k-m}(Z)$. In the next chapter we establish further properties of this functor.
\end{remark}

\subsection{Chern intersection classes}
In this section we define Chern intersection functors on the rational Chow categories in two fashions; via rigidity and via using the filtration $\op{Fil}_k$ introduced in the previous section. Finally we compare it to the already constructed Chern intersection functors given by Franke in \cite{AAG}. We will suppose $Z$ has the same properties as the in previous section, \ie that $Z$ is a separated scheme of finite type over a regular scheme moreover admitting an ample family of line bundles. For example, any quasi-projective scheme over a regular scheme. \\

Recall that, by the projection formula the space $((\bb Z \times \op{Gr})_\bb Q)^n$ represents the functor sending a regular $X$ to $K_0(\bb P^n_X)_\bb Q$. Indeed, if $p: \bb P^n_X \to X$ is the projection there is the natural isomorphism $$K_0(\bb P^n_X) \simeq \bigoplus_{i=0}^{n-1} \calo(-i) \otimes Lp^* K_0(X).$$
Denote now by $c_1(\calo(j))_i \cap - = c_1(\calo(j)) \cap $ the homomorphism projecting multiplication by $\calo(j)$ on $K_0(X)^{(i)}$ to $K_0(X)^{(i+1)}$ and $c_1(\calo(j))^k$ the iterated homomorphism $k$ times. Then there is a decomposition of Adams-eigenspaces $$K_0(\bb P^n_X)^{(i)} = \bigoplus_{j=0}^{n-1} c_1(\calo(-1))^j \cap Lp^* K_0(X)^{(i-j)}.$$
This is a consequence of \cite{Riemann-RochAlgebra}, Chapter III, Theorem 1.2 and a comparison with the usual first Chern class is well-known. 
\mps{One moreover has the equivalence of functors $c_1(\calo(i+j)) \cap - = c_1(\calo(i)) \cap - + c_1(\calo(j)) \cap -$. In particular from the Segre embedding we deduce that \begin{equation}\label{Chern-additivity} \textbf{c}'_1(\calo_{\bb P^{nm-m+n}}(1)) \cap  = \textbf{c}'_1( p_1^* \calo_{\bb P^{n}}(1)) \cap  +  \textbf{c}'_1(p_2^* \calo_{\bb P^{m}}(1)) \cap \end{equation} canonically, where $p_1$ and $p_2$ are the two projections. } \begin{D} \label{defn:rigidfirstchern} Thus by rigidity we obtain an intersection functor $$\textbf{c}'_1(\calo_{\bb P^n}(1)) \cap -: \mc {CH}^i_{\bb P^n_Z}(\bb P^n_X) \to \mc {CH}^{i+1}_{\bb P^n_Z}(\bb P^n_X).$$ Given a scheme $Z$ in a regular scheme $X$ and a surjection $\calo^{n+1} \stackrel {\phi} \longrightarrow L$ where $L$ is a line bundle on $Z$, there is an induced functor, where we use that $\mc {CH}_{k+n}(\bb P^n_Z) = \mc {CH}^{\dim X - k}_{\bb P^n_Z}(\bb P^n_X)$ and $\mc {CH}_{k+n-1}(\bb P^n_Z) = \mc {CH}^{\dim X - k +1 }_{\bb P^n_Z}(\bb P^n_X)$,
$$\textbf{c}'_1(L,\phi) \cap -: \mc{CH}_k(Z) \stackrel{p^*} \longrightarrow \mc {CH}_{k+n}(\bb P^n_Z) \stackrel{\textbf{c}'_1(\calo_{\bb P^n}(1)) \cap -} \longrightarrow \mc {CH}_{k+n-1}(\bb P^n_Z) \stackrel{s^*} \longrightarrow \mc {CH}_{k-1}(Z) $$
where $s$ is the section to $p: \bb P^n_Z \to Z$ induced by $\phi$. We call it the rigid first Chern class-functor, induced by $\phi$. \end{D}

\begin{D} \label{defn:firstchernclass} The bi-additive functor $\textbf{c}_1(-) \cap -: \hbox{Vect}_Z \times \mc {CH}_i(Z) \to \mc {CH}_{i-1}(Z)$ is defined as follows. For a fixed vector bundle $E$, $\textbf{c}_1(E) \cap -$ is the functor $$\mc {CH}_i(Z) \stackrel{\pi_{n-i}} \longrightarrow  G(Z) \stackrel{E \otimes -} \longrightarrow G(Z) \stackrel{p_{n-i+1}}\longrightarrow \mc {CH}_{i-1}(Z)$$ where the non-trivial functors are given by the natural inclusions and projections. We call it the first Chern class-functor.
\end{D}
\begin{Prop} \label{Prop:Chernfunctorcomparison} The rigid first Chern class-functor does not depend on the particular choice of $X$ and $\phi$ and when $E$ is a line bundle $L$ both can be identified with the induced intersection functor $$(L-1) \cap - : \mc {CH}_i(Z) \to \mc {CH}_{i-1}(Z)$$
from the proof of Proposition \ref{prop:projectionformulafiltration}.
\end{Prop}
\begin{proof}
The second part follows from the fact that $1$ acts trivially on the virtual category. To prove that the functor is not dependent on the choice of $X$ and $\phi$, it is enough to prove that the two functors coincide for $\calo(i)$ over $\bb P^n_Z$ since the first Chern class-functor does not involve either and is compatible with projection formula together with that  $p \circ s = \op{id}$ where $s$ is the section induced by $\phi$ in Definition \ref{defn:rigidfirstchern}. Now, in this case we can describe both intersection functors explicitly. The rigid functor is characterized by acting on the class $[\bb P^k_Z]$ where $\bb P^k_Z \subseteq \bb P^n_Z$ is a linear subspace by restricting to $\bb P^{k-1}_Z = \bb P^{k}_Z \setminus \bb A^k_Z$ on the level of Chow groups. But the same is true for multiplication by $(L-1)$ by the description in the proof of Proposition \ref{prop:projectionformulafiltration} and it is functorial with respect to basechange. This obviously also holds when restricting to the case $Z = X$ and $X$ a Grassmannian and thus by rigidity they must coincide in general.

\end{proof}

\begin{T} \label{thm:firstchernclassproperties} Let $Z$ be a scheme as in the beginning of this section and $A$ an object of $\mc{CH}_k(Z)$. The Chern intersection functors satisfy the following properties.
\begin{description}
  \item[(i)] [Projection formula] Suppose $f: Z \to Z'$ is a proper morphism of schemes embeddable into regular schemes, and $L$ is a line bundle on $Z'$, there is a functorial isomorphism of functors in $\mc{CH}_{k-1}(Z')$
      $$f_*(\textbf{c}_1(f^* L) \cap A) \simeq \textbf{c}_1(L) \cap f_* A.$$
  \item[(ii)] [Additivity] If $L$ and $M$ are line bundles on $Z$, there is a canonical isomorphism
  $$\textbf{c}_1(L \otimes M) \cap - \simeq \textbf{c}_1(L) \cap A + \textbf{c}_1(M) \cap A$$
  which is commutative.
  \item[(iii)] [Commutativity] If $L$ and $M$ are line bundles on $Z$, then there is a canonical isomorphism
  $$\textbf{c}_1(L) \cap (\textbf{c}_1(M) \cap -) \simeq \textbf{c}_1(M) \cap (\textbf{c}_1(L) \cap A) $$
  functorial in $A$.

\end{description}
\end{T}
\begin{proof} For $(i)$, it follows by Proposition \ref{prop:projectionformulafiltration} and Proposition \ref{Prop:Chernfunctorcomparison}. Thus by $(i)$we can assume that $A$ is the unit object for the other properties. For $(ii)$, the isomorphism follows from the the fact that $(L-1)\otimes (M-1) \cap -$ defines a functor $\op{Fil}_k(Z) \to \op{Fil}_{k-2}(Z)$ and the isomorphism $LM - L - M = (LM - 1) - (L-1) - (M-1) = (L-1)\otimes (M-1)$ which is canonical by the fact that we have inverted two (cf. \cite{determinant}, 9.7.4). The commutativity follows from the isomorphism $(L-1) \otimes (M-1) \simeq (M-1) \otimes (L-1)$ which is canonical in view of that we work with rational coefficients and in particular having inverted $2$. Finally for $(iii)$, suppose first that we are given two rational section $l$ (resp. $m$) of $L$ (resp. $M$) so that $L$ (resp. $M$) is isomorphic to $\calo(\op{div} l)$ (resp. $\calo(\op{div} m))$. If they are given by irreducible effective divisors which meet properly so that their scheme-theoretic intersection is of codimension 2 they define the same element in $\mc{CH}_{k-2}(Z)$ by the calculation in \cite{intersection}, Theorem 2.4 and the general case done in ibid. Given another set of rational sections $l'$ and $m'$ the two objects are isomorphic as is also easily verified and we conclude.  
\end{proof}

\begin{TL} \label{lemma:firstchernclass-det} There is a canonical isomorphism $\textbf{c}_1(E) \cap - \simeq \textbf{c}_1(\det E) \cap -$ compatible with, for an exact sequence of vector bundles $0 \to E' \to E \to E'' \to 0$,  $$\textbf{c}_1(\det E' \otimes \det E'') \cap - = \textbf{c}_1(\det E') \cap - \oplus \textbf{c}_1(\det E'') \cap -$$ induced by the canonical isomorphism $\det E \simeq \det E' \otimes \det E''$.
\end{TL}
\begin{proof} The proof of the corresponding theorem in "Chern Functors" in \cite{AAG}, 1.13.2 carries over to this situation. To define the isomorphism we can use any complete flag on $E$ and Theorem \ref{thm:firstchernclassproperties}, $(ii)$. We need to verify that this is not dependent on the choice of flag. By the projection formula we can assume that the virtual coherent sheaf $A$ is the trivial sheaf. In this case the default of the required commutativity defines an element of $\op{Aut}(\mc{CH}_{k-1}(Z))$ for $Z$ equidimensional of dimension $k$, stable under basechange. Restricting to the generic point of $Z$ is flat and fully faithful and by pullback we can furthermore suppose $Z = \spec K$ for a field $K$ so that the element in question is an element of $K^*$. Since the complete flag variety is proper over $K$ the field of global invertible functions identifies with $K^*$ and thus the isomorphism necessarily descends.
\end{proof}
\begin{Prop} Suppose that $X$ is regular. The construction of the first Chern intersection functor here is naturally equivalent with that of \cite{ChowCT} under the the (trivial) equivalence of categories in Proposition \ref{prop:ChowCatequiv} (b).
\end{Prop}
\begin{proof} This is completely formal and we will only recall the main properties of the other construction of the first Chern intersection functors needed for the comparison. Without loss of generality we can assume $X$ is connected. By Lemma \ref{lemma:firstchernclass-det} we can suppose the vector bundle $E$ is given by a line bundle $L$. By Corollary \ref{Cor:homotopyinvariance} and Jouanolou-Thomason (cf. \cite{Weibel}, Proposition 4.4) we can assume that $X$ is moreover affine. Then we can suppose we are given a surjection $\calo^{n+1} \twoheadrightarrow L$ for some $n$ and the induced pullback to $\bb P^n_X$ is faithful so that we can suppose $L = \calo(1)$. Moreover by the projective bundle formula and the projection formula we only need to compare the two functors evaluated on the trivial bundle. Since $X$ is integral it is either flat over $\spec \bb F_p$ or $\spec \bb Z$. It thus pulls back from the arithmetic situation via the natural map (cf. Remark \ref{remark:fiberedPicard-flat}). In this case the equality is tautological since $\mc {CH}_k(\bb P^n_\bb Z)$ and $\mc {CH}_k(\bb P^n_{\bb F_p})$ have no non-trivial automorphisms and the two operations coincide with the usual Chern intersection classes on the level of isomorphism classes.
\end{proof}
This also proves that the constructions coincide with those of \cite{Elkik}, since they coincide by a remark in \cite{AAG}. The method of "Chern Functors" in \ibid, 1.13.2, reminiscent of Grothendieck's construction of Chern classes using the projective bundle formula for Chow groups, provides the construction of functors $\textbf{c}_j(E) \cap -:  \mc {CH}_i(Z) \to \mc {CH}_{i-j}(Z)$. By unicity of Chern functors (ibid.) we obtain the following:
\begin{Cor} There are natural Chern intersection-functors, for any vector bundle $E$, $$\textbf{c}_j(E) \cap -:  \mc {CH}_i(Z) \to \mc {CH}_{i-j}(Z)$$
satisfying, for a short exact sequence $0 \to E' \to E \to E'' \to 0$ we have an equivalence of functors
$$\textbf{c}_j(E) \cap - \simeq \oplus_{i = 0}^j \textbf{c}_i(E') \cap \textbf{c}_{j-i}(E'') \cap -  $$
and they are isomorphic to the functors in "Chern Functors" in \cite{AAG}, 1.13.2.
\end{Cor}
\begin{Cor} The Chern intersection functors coincide with the functors determined by $\gamma$-operations, \ie the functor
$$\textbf{c}_j(E) \cap -:  \mc {CH}_i(Z) \to \mc {CH}_{i-j}(Z)$$
is canonically equivalent to the functor
$$\gamma^j({E - \rk E}) \cap -: F_i(Z)/F_{i-1}(Z) \to F_{i-j}(Z)/F_{i-1-j}(Z)$$
determined for $j = 1$ or $j = \rk E$ in general and for other $j$ by Corollary \ref{Cor:gamma} whenever $Z$ is regular.
\end{Cor}
\begin{proof} This follows from the splitting principle and the fact that that they coincide for line bundles.
\end{proof}
\appendix
\section{$\bb A^1$-homotopy theory of schemes}  \label{B}

This section is to recall some necessary results and and to fix some notation. In what follows we have but slight extensions of the theorems in the reference-list, and we hope the reader agrees that not spelling out the proofs does not cause any harm. One word of warning though, we have almost completely ignored issues related to smallness of categories. This can be amended by inserting the word "universe" at the appropriate places.

%
%
Denote by $\Delta$ the category of
totally ordered finite sets and monotonic maps. Hence, the objects
are the finite sets $[n] = \{ 0 < 1 < 2 < \ldots < n \}$ and the
morphisms of $\Delta$ are generated by the maps

$$\delta_i: [n-1] \to [n],  \hbox{ defined by } \delta_i(j) = \bigg \{
  \begin{array}{ll}
      j,  & \hbox{if } j < i \\
      j + 1, & \hbox{if } j \geq i
  \end{array}  $$ and
$$\sigma_i: [n] \to [n-1],  \hbox{ defined by }  \sigma_i(j) =  \bigg \{
  \begin{array}{ll}
      j,  & \hbox{if } j \leq i \\
      j - 1, & \hbox{if }  j > i
  \end{array} $$

These maps are the face resp. the degeneracy-maps, and
satisfy the usual simplicial relationships (\cite{SimplicialHT}, chapter 1). If $\mc C$ is any category, we denote by $s \mc C$ or
$\Delta^{\operatorname{op}} \mc C$ the category of simplicial objects of $\mc C$, \ie the category whose objects are functors
$\Delta^{\operatorname{op}} \rightarrow \mc C$, and morphisms
are natural transformations of functors. \\
Let $T$ be a site, and denote by $\operatorname{\mathbf{Shv}}(T)$
the category of sheaves of sets on $T$, and
$\Delta^{\operatorname{op}}\operatorname{\mathbf{Shv}}(T)$ the
category of simplicial sheaves. Note that if we are given a
simplicial set $E$, we can associate to it the constant simplicial
sheaf, which we also denote by $E$, and thus we obtain a functor
$$\Delta^{\operatorname{op}} \operatorname{Set} \stackrel{\small{\hbox{constant}}}{\longrightarrow} \Delta^{\operatorname{op}} \mathbf{Shv} (T).$$
The standard $n$-simplices $\Delta^n$ define thus by the Yoneda
lemma a cosimplicial object

$$\xymatrix{
\Delta & \stackrel{{\Delta^\bullet}}{\longrightarrow} & \Delta^{\operatorname{op}}\operatorname{\mathbf{Shv}}(T) \\
n & \mapsto &\Delta^n}$$

and we give the category
$\Delta^{\operatorname{op}}\operatorname{\mathbf{Shv}}(T)$ the
structure of a simplicial category with a simplicial function object
$\fhom(-,-)$ given by
$$\fhom(\mc X, \mc Y) := \Hom_{\Delta^{\operatorname{op}} \operatorname{\mathbf{Shv}}(T)}(\mc X \times \Delta^\bullet, \mc Y).$$
Before continuing, we recall the fundamental lemma of homotopical
algebra \begin{T} \label{thm:fundhomotopicalalgebra}[\cite{SimplicialHT}, II.3.10] Let $\mc C$ be a closed simplicial model category with
associated homotopy-category $\mc H$, and $\mc X, \mc Y \in
\op{ob}(\mc C)$. Suppose furthermore that $\mc X' \to X$ is a
trivial fibration with $\mc X'$ cofibrant and $\mc Y \to \mc Y'$ is
a trivial cofibration with $\mc Y'$ fibrant. Then we have a natural
identification
$$\hom_\mc H(\mc X, \mc Y) = \pi_0(\fhom(\mc X', \mc Y')).$$
\end{T}

An adjoint to the functor $\Delta^{\operatorname{op}}
\operatorname{Set} \to \Delta^{\operatorname{op}} \mathbf{Shv} (T)$
given by $X \mapsto \fhom(\ast, X),$ which we sometimes write as $X
\mapsto |X|$.

\D{Let $f: \mc X \to \mc Y$ be a morphism of simplicial pre-sheaves. Then \\
\begin{enumerate}
 \item $f$ is said to be a (simplicial) weak equivalence if, for any conservative
family $\{x: \operatorname{\mathbf{Shv}}(T) \to \operatorname{Set}
\}$ of points of $T$, $x(f): x(\mc X) \to x(\mc Y)$ is a
homotopy-equivalence of simplicial sets.
 \item $f$ is called a cofibration if it is a monomorphism.
 \item $f$ is called a fibration if it has the right lifting property
with respect to trivial cofibrations, \ie cofibrations which are
also weak equivalences. \end{enumerate} }

\begin{T} \label{thm:modelcategoryI}[\cite{VoevMorel}, Theorem 2.1.4] For any (small) site with enough points $T$, the above equips
$\Delta^{\operatorname{op}}\operatorname{\mathbf{Shv}}(T)$ with the
structure of a closed model category.
\end{T}

We denote by $\mc H_s(T)$ the corresponding homotopy-category
obtained by inverting the weak equivalences in
$\Delta^{\operatorname{op}} \operatorname{\mathbf{Shv}}(T)$. To fix ideas, unless explicitly mentioned, from here on $S$
will denote a regular scheme and $T$ a full
subsite with enough points of $Sch/S_{sm}$ the category of $S$-schemes
equipped with the smooth topology \footnote{\ie a full subcategory such that any open cover of $T$ is an open cover of $Sch/S_{Sm}$.}
, and denote the corresponding homotopy-category by $\mc H_s(T)$. Most often, we will be concerned with the category $\mathfrak R_S$
of regular $S$-schemes with the smooth topology. When $S = \spec \bb Z$, we write $\mathfrak R_\bb Z = \mathfrak R$. Since any smooth morphism locally for the étale topology has a section we can identify the various
topoi of sheaves of regular $S$-schemes with étale or smooth topology or of affine regular $S$-schemes with the étale or smooth topology with a "big regular
étale $S$"-topoi. They are given a conservative set of points by regular local strict henselian rings.
\D{\label{defn:Bousfield}
Suppose $T$ is such that for any $X \in ob(T)$, $\bb A^1_X$ is also
an object in $T$. We say that $\mc X \in \mc H_s(T)$ is $\bb
A^1$-local with respect to $T$, if for any $\mc Y \in
\operatorname{\mathbf{Shv}}(T)$ the map
$$\Hom_{\mc H_s(T)}(\mc Y \times \bb A^1, \mc X) \to \Hom_{\mc H_s(T)}(\mc Y, \mc X)$$
is bijective. We say a morphism $f: \mc X \to \mc Y$ in
$\Delta^{\operatorname{op}}\operatorname{\mathbf{Shv}}(T)$ is $\bb
A^1$-local if for any $\bb A^1$-local object $\mc Z$, the natural
map $$\Hom_{\mc H_s(S)}(\mc Y, \mc Z) \to \Hom_{\mc H_s(T)}(\mc X,
\mc Z)$$ is bijective.} Now equip
$\Delta^{\operatorname{op}}\operatorname{\mathbf{Shv}}(T)$
with $\bb A^1$-local weak equivalences, cofibrations and $\bb
A^1$-local fibrations. Then we have:
\begin{T} [\cite{VoevMorel}, Theorem 2.3.2] This equips
$\Delta^{\operatorname{op}}\operatorname{\mathbf{Shv}}(T)$ with the
structure of a closed model-category. \end{T} \D{We denote the corresponding homotopy category by $\mc H(T)$. Whenever $T=
Sm/S_{Nis}$, the corresponding homotopy-category is the $\bb
A^1$-homotopy category of schemes over $S$ defined by loc.cit., but it will not directly play
a role in what we do.
When the site is $T = \mathfrak R_{sm}$, the category of regular schemes with the smooth topology,
the corresponding homotopy category is denoted by $\mc H(\mathfrak R)$. We also have natural pointed analogues.
Replacing in all previous definitions pointed versions, we obtain the $\bb A^1$-homotopy category of pointed
simplicial sheaves $\mc H_\bullet(T)$ as a localization of the
category of pointed simplicial sheaves;
$\Delta^{\operatorname{op}}\operatorname{\mathbf{Shv}}(T)_\bullet$.}
For two objects $(\mc X,x), (\mc Y,y) \in \mc H_\bullet(T)$ we
define $X \wedge Y$ in the usual way as the coequalizer of $$\mc X
\times y, x \times \mc Y \rightrightarrows \mc X \times \mc Y.$$ For
a simplicial sheaf $\mc X$, we denote by $\mc X_+$ the simplicial
presheaf with a disjoint point. The functor $\mc X \to \mc X_+$ is
left adjoint to the forgetful functor $\mc H_\bullet(T) \to \mc
H(T)$.

The stable homotopy-category of schemes is stabilized out of the
"unstable" one with the proper notion of a circle. As before, let $T$ denote
a (small) site with enough points.

\D{Let $\mathbf T \in \Delta^{\operatorname{op}}\operatorname{\mathbf{Shv}}(T)_\bullet$.
A $\mathbf T$-spectra is a set $\mathbf E = (d_n, \mathbf{E}_n)_{n
\in \bb N}$ of objects in
$\Delta^{\operatorname{op}}\operatorname{\mathbf{Shv}}(T)_\bullet$
 with morphisms
 $$d_n: \mathbf T \wedge \mathbf  E_n \to \mathbf E_{n + 1}.$$
 A morphism of $\mathbf T$-spectra $f:\mathbf E \to \mathbf F$ is a set of morphisms $f_n: \mathbf E_n \to \mathbf F_n$
 such that the diagram commutes
$$ \xymatrix{\mathbf T \wedge E_n \ar[r] \ar[d] & E_{n+1} \ar[d] \\
 \mathbf T \wedge F_n \ar[r] & F_{n+1} }.$$}

\D{Let $\mathbf E$ be a $\mathbf T$-spectra, and denote by $\Omega_\mathbf T(-) = \Omega(-) = R\Hom(\mathbf T, -)$ the total derived functor
(in $\mc H_\bullet(S)$) of the right adjoint to $\mathbf T \wedge
-$. We say that $\mathbf E$ is a $\Omega$-spectra if for any $n$ the induced
morphism
$$\mathbf E_n \to \Omega (\mathbf E_{n + 1}) $$ is in fact
an isomorphism. } We can naively construct "a" stable homotopy-theory
by taking the category of $\Omega$-spectras with respect to
$\mathbf{T}= (\bb P^1, \infty)$, and denote it by $\mc
{SH}_{naive}(T)$, and giving morphisms $E \to F$ by morphisms $E_n \to F_n$ in $\mc H_\bullet(S)$ for any $n$ such that the obvious diagram commutes (cf. \cite{Riou}, Définition I.124).

\D{Let be a morphism $f: \mathbf E \to \mathbf F$ of $\mathbf T$-spectras.
Then $f$ is a projective cofibration if $f_0$ is a monomorphism and
for any $n > 0$, $$\mathbf T \wedge \mathbf F_n \bigvee_{\mathbf T
\wedge \mathbf E_n} \mathbf E_{n +1} \to \mathbf F_{n+1}$$ is also a
monomorphism. Its an $\bb A^1$-projective fibration (resp. $\bb
A^1$-projective equivalence) if every map $f_n$ is a $\bb
A^1$-fibration (resp. $\bb A^1$-weak equivalence).}

\begin{T} [\cite{Riou}, Première partie] Let $\mathbf T = (\bb P^1, \infty)$. The category of $\mathbf T$-spectras equipped with projective cofibrations as cofibrations, $\bb A^1$-projective fibrations
as fibrations and $\bb A^1$-projective equivalences as weak
equivalences is a closed model-category.
\end{T}
We define the following as the stable homotopy-category of $T$.

\D{ Let $\mathbf T = (\bb P^1, \infty)$. Then the stable homotopy-category $\mc {SH}(T)$ is
the full subcategory, of the corresponding homotopy-category, of
$\Omega$-spectras.}

\D{\label{Grassmanniandefn} For a fixed scheme $S$, let
$\op{Gr}_{d,r}$ be the Grassmannian of locally free quotients of
rank $r$ of $\calo_S^{d+r}$ viewed as an object of $\op{Shv}(T)$. Notice that $\op{Gr}_{d,r} \simeq \op{Gr}_{r,d}$. Let $\mc F$ be a locally free sheaf of rank $r$. We have natural morphisms $\op{Gr}_{d,r} \to \op{Gr}_{d+1,r}$ and
$\op{Gr}_{d,r} \to \op{Gr}_{d,r+1}$ by sending $\phi: \calo^{d+r} \twoheadrightarrow \mc F$ to $\calo^{d+r+1} \stackrel{(\phi,0)}\twoheadrightarrow \mc F$ and
$\calo^{d+r+1} \stackrel{\phi, \op{id}}\twoheadrightarrow \mc F \oplus \calo$ respectively. We denote by $\op{Gr}_d = \lim_{\rightarrow} \op{Gr}_{d,r}$ and $\op{Gr} = \lim_{\rightarrow}
\op{Gr}_d$ for these maps. Here the direct limits are taken in $\operatorname{\mathbf {Shv}}(T)$. Since all things here naturally pointed (by $\op{Gr_{d,0}}$ for any $d$),
we also obtain a pointed element $\op{Gr} \in \mc H_\bullet(T)$. Notice that $\bb P^d = \op{Gr}_{d,1} \simeq \op{Gr}_{1,d} $ and denote by $\bb P^\infty = \op{Gr_1}$.}

By the method of \cite{Riou}, Définition III.101, it is possible to define a
sheaf $(\bb Z \times \op{Gr})[\frac{1}{n}]$ and $(\bb Z \times \op{Gr})_\bb Q$ with a natural morphism $\bb Z \times \op{Gr} \to (\bb Z \times \op{Gr})[\frac{1}{n}]$
and $\bb Z \times \op{Gr} \to (\bb Z \times \op{Gr})_\bb Q$. In a similar fashion to \loccit, to lax notation first put $\op{Gr}_{d,r} = \op{Gr}^{d+r,r}$ so that $\bb P^d = \op{Gr}^{d+1, 1}$ and define a morphism $m_{a,d}: \op{Gr}^{d,1} \to \op{Gr}^{d^a,1} $
by sending a surjection $p: \calo^{d} \twoheadrightarrow \mc L$ to $p^{\otimes a}: (\calo^{d})^{\otimes a} \twoheadrightarrow \mc L^{\otimes a}$. One verifies the relation
$$\xymatrix{\op{Gr}^{d,1}\ar[rrr]^{m_{a,d}} \ar[d] & & & \op{Gr}^{d^a,1} \ar[d] \\
\op{Gr}^{d+1,1} \ar[rrr]^{m_{a,d+1}} & & & \op{Gr}^{(d+1)^a,1} = \op{Gr}^{(d^a + [(d+1)^a - d^a],1}}$$
and define $m_a: \bb P^\infty \to \bb P^\infty$ to be the induced morphism. The relation $m_{ab} = m_a m_b$ is easy. \footnote{To make the above a proper definition and make the diagram commute on the nose, one needs to define a natural isomorphism $\delta_{a,d}: (\calo^d)^{\otimes a} \to \calo^{d^a}$. It can be done as follows. We define a strict total order on $\{ e_1^1 , e_2^1 , \ldots e_d^1 \} \times \ldots \times \{ e_1^a , e_2^a , \ldots e_d^a \}$, \ie the structure of the category $[n^a-1]$ inductively as follows. First $e_{i_1}^1 \times e_{i_2}^2 \times \ldots \times e_{i_a}^a < e_{j_1}^1 \times e_{j_2}^2 \times \ldots \times e_{j_a}^a $: If  $\op{max} i_k< \op{max} j_l$. If there is equality $i_m = \op{max} i_k = \op{max} j_l = j_n$ , then if $\op{max} i_k \setminus i_n < \op{max} j_l \setminus j_m$. Repeatedly removing such $m,n$'s we obtain an order on all objects except when the $i_k$ are a permutation of the $j_l$'s. With these, pick the lexicographic order. We then define an isomorphism $\delta_{a,d}: (\calo^d)^{\otimes a} \to \calo^{d^a}$ by sending a basis-element $e_{i_1}^1 \times e_{i_2}^2 \times \ldots \times e_{i_a}^a$ to the basis $f_i$ with $i \in [n^a-1]$ via the ordering just constructed. As such, the element $e_{1}^1 \times e_{1}^2 \times \ldots \times e_{1}^a$ is the smallest element, $e_{1}^1 \times e_{2}^2 \times e_{1}^3 \times \ldots \times e_{1}^a < e_{2}^1 \times e_{1}^2 \times e_{1}^1 \times \ldots \times e_{1}^a$ and $e_{d}^1 \times e_{1}^2 \times \ldots \times e_{1}^a$ is generally a big element.}
\D{\label{rationalgrassmannian} One defines $\bb P^\infty[\frac{1}{n}]$ (resp. $\bb P^\infty_\bb Q$) as the inductive limit over $m_a$'s ordered by division for $a = n^k, k \in \bb N$ (resp. $m_a$'s ordered by division for all $a \in \bb N$).}

One of the main observations of \cite{VoevMorel} is the following theorem, which states that algebraic $K$-theory is
represented by an infinite Grassmannian. The version presented below is proven in exactly the same way as in the article in question, with the exception of
using smooth descent for rational $K$-theory instead of Nisnevich descent. Note that since any smooth morphism locally for the étale topology has a section \mps{"Note that since any smooth morphism locally for the étale topology has a section" was added}we have étale
descent whenever we have smooth descent, and in the former case the statement we are looking for is \cite{Thomason2}, Theorem 11.11;

\begin{T}[\cite{VoevMorel}, Theorem 4.3.13] \label{thm:voevmorelGr} Let $S$ be a regular scheme. Then we have
canonical functorial isomorphisms
$$\Hom_{\mc H_\bullet(\mathfrak R_{S,sm})}(S^n \wedge X_+, (\bb Z \times \op{Gr})_\bb Q) =
\Hom_{\mc H(\mathfrak R_{S,sm})}(X, \Omega^n \left( \bb Z \times \op{Gr}\right)_\bb Q) \simeq K_n(X)_\bb Q$$
for $X$ a regular $S$-scheme, where $K_n$ refers to Quillen's $K$-theory defined
as above. In particular, we have an isomorphism
$$\Hom_{\mc H(\mathfrak R_{S, sm})}(X, (\bb Z \times \op{Gr})_\bb Q) \simeq K_0(X)_\bb Q.$$
\end{T}
Proceeding as in \cite{Riou}, Chapitre III, one constructs a product $$(\bb Z \times \op{Gr})_\bb Q \wedge (\bb Z \times \op{Gr})_\bb Q \to
(\bb Z \times \op{Gr})_\bb Q$$ in $\mc H_\bullet(\mathfrak R_{S,sm}).$
\begin{Prop} Consider the natural map $t: \bb P^1 \to \{ 0 \} \times \op{Gr} \to (\bb Z \times \op{Gr})_\bb Q$. Then the data $\mathbf E = (\mathbf  E_i, d_i)$
defined by $\mathbf E_i = \bb Z \times \op{Gr}$ and the product
$$d_i: \bb P^1 \wedge (\bb Z \times \op{Gr})_\bb Q \stackrel{t \wedge \op{id}}\to (\bb Z \times \op{Gr})_\bb Q \wedge (\bb Z \times \op{Gr})_\bb Q \to
(\bb Z \times \op{Gr})_\bb Q$$ is a naive spectrum, which we denote by $K_{naive}$.
\end{Prop}
\begin{proof} We need to show that the natural map $$(\bb Z \times
\op{Gr})_\bb Q \to R\Hom_\bullet ((\bb P^1, \infty), (\bb Z \times \op{Gr})_\bb Q)$$
is an isomorphism. However, this follows from the fact that for any
$S$-scheme $X$, the map $$K_n(X) \to \{ y \in K_n(\bb P^1_X), \infty^*
y = 0 \in K_n(X)\}$$ given by $x \mapsto x \boxtimes u$, where $u = \calo(1) - 1$, is bijective, which in turn is a consequence of
the projective-bundle-formula for $K$-theory.
\end{proof}

Notice there is an obvious forgetful functor
$$\mc {SH}(T) \to \mc {SH}_{naive}(T).$$ There are a priori several
liftings of the naive spectrum constructed above representing
algebraic $K$-theory, and we make the following definition:

\label{defn:stablemodel} \D{A stable model for rational algebraic
$K$-theory is an object $\mc {K} \in \mc {SH}(T)$, together with an
isomorphism $$\omega: forget(\mc K) \simeq (K_{naive})_\bb Q.$$ }

\begin{T}[\cite{Riou}, Chapitre V] When $S = \spec \bb Z$ and $T = \mathfrak R_{sm}$ there is a unique, up to
unique isomorphism, stable model for rational algebraic $K$-theory, denoted
by $\mathbf{BGL}_\bb Q$, thus defining a canonical rational stable model $\mc SH(\mathfrak R_{S,sm})$ for any
regular scheme $S$ via $\mathbf{BGL}_{S,\bb Q} := f^* \mathbf{BGL}_\bb Q$ where $f: S \to \spec \bb Z$ is the natural morphism.
\end{T}

Let $S$ be a Noetherian, regular scheme. By the Yoneda lemma, we
have a functor
$$\Phi: T \to \operatorname{\mathbf{Shv}}(T) \to \Delta^{\operatorname{op}} \operatorname{\mathbf{Shv}}(T) \to \mc H(T). $$
If $G$ is any object of $\mc H(T)$, we denote by $\phi G$ the
presheaf on $T$ defined by
$$T \ni U \mapsto \Hom_{\mc H(T)}(\Phi U, G).$$
In particular, we have an isomorphism $$\phi(\bb Z \times \op{Gr})_\bb Q
\simeq K_0(-)_\bb Q.$$
\begin{T}[Théorème III.29 in \cite{Riou}] Let $S$ be a regular scheme. Given two (pointed) presheaves $\mc F, \mc G$
on $\mathfrak R_S$ denote by $\Hom_{\mathfrak R_S^{\operatorname{op}} \op{Set} }(\mc F, \mc G)$ (resp.
$\Hom_{\bullet, \mathfrak R_S^{\operatorname{op}} \op{Set}}(\mc F, \mc G)$) the set of (pointed) natural transformations
from $\mc F \to \mc G$. Then the natural morphism
$$\Hom_{\mc H(\mathfrak R_{S, sm})}((\bb Z \times \operatorname{Gr})_\bb Q, (\bb Z \times \operatorname{Gr})_\bb Q) \to \Hom_{\mathfrak R_S^{\operatorname{op}} \op{Set}}(K_0(-)_\bb Q, K_0(-)_\bb Q)$$
(resp.
$$\Hom_{\mc H_\bullet(\mathfrak R_{S, sm})}((\bb Z \times \operatorname{Gr})_\bb Q, (\bb Z \times \operatorname{Gr})_\bb Q) \to \Hom_{\bullet, \mathfrak R_S^{\operatorname{op}} \op{Set}}(K_0(-)_\bb Q, K_0(-)_\bb Q))$$
is bijective.
\end{T}

\begin{T} \label{thm:BGLdecomposition}[Théorème IV.72 in \cite{Riou}] We have a natural decomposition in terms of
"Adams eigenspaces",
$$\mathbf{BGL}_\bb Q \simeq \bigoplus_{i \in \bb Z} \bb H^{(i)}$$
In $\mc SH(\mathfrak R_S)$.
\end{T}
Since the above theorem will play a crucial role, we will mention
the highlights of the proof. \\
First, put
$$p_n = \frac{1}{n!}\log^n(1 + U) \in \bb Q[[U]],
$$ and
$$\sum k^n p_n = (1 + U)^k = \Psi^k.$$
This element will play the role of the $k$-th Adams operator. Let $A$ be an abelian group and define the operator $\Omega$ on $A[[U]]$ by the formula
$$\Omega(f) = (1+U)\frac{df}{dU}$$
and by $A^\Omega$ the following inverse limit:
$$\ldots \stackrel{\Omega}\to  A[[U]] \stackrel{\Omega}\to A[[U]] \stackrel{\Omega}\to A[[U]].$$
 Now, by \cite{Riou}, Theorem IV. 55, then we have
bijections
$$\Hom_{\mc{SH}(S)}(\mathbf{BGL}, \mathbf{BGL}_\bb Q) \simeq \lim (K_0(S)_\bb Q)^\Omega,$$
and there is a natural map $$\bb Q^\bb Z \to \lim (K_0(S)_\bb
Q)^\Omega$$ induced by a commutative diagram (whenever $A$ is a $\bb
Q$-vector space, see loc.cit. Lemme IV. 66):
$$\xymatrix{A^\bb N \ar[r]^\sigma \ar[d]^s &  A[[U]] \ar[d]^\Omega \\
A^\bb N \ar[r]^\sigma & A[[U]]}$$ where $s((a_n)_{n \in \bb N}) =
(a_{n+1})_{\bb N}$ and $\sigma((a_n)_{n \in \bb N)}) = \sum_{n \in
\bb N} a_n p_n.$ Here this induces an isomorphism $\Sigma: A^\bb Z
\simeq \lim A^\Omega$ defined by $[(a_i)_{i \in \bb Z}]_{i \in \bb
Z} \mapsto \sum a_{i+n} p_n \in A[[U]]$(loc.cit. Corollaire IV.67),
and this is the map alluded to above. \\
Composing the above two maps, we obtain a map
$$\bb Q^\bb Z \to \End_{\mc {SH}(\mathfrak R_S)}(\mathbf{BGL}_\bb Q).$$
Consider the characteristic function $\pi_i$ associated to $\{i\}$
viewed as an element $\bb Q^\bb Z$. It corresponds naturally to the
element $(p_{i + n})_{n \in \bb N}$ in $\lim \bb Q^\Omega$ (where we
put $p_k = 0$ for $k < 0$). They naturally define an orthogonal
family of idempotents in $\End_{\mc{SH}(S)}(\mathbf{BGL}_\bb Q).$
$\mc {SH}(\mathfrak R_S)$ being an pseudo-abelian category, we can consider the image of $\pi_i$, and
we denote it by $\bb H^{(i)}$. It is formal that $\Psi^k$ above acts
as multiplication by $k^i$ on $\bb H^{(i)}$.

\begin{Cor} [\cite{Riou}, Corollaire IV.75] We have a decomposition $$K_i(X)_\bb Q = \bigoplus_{j \in \bb N} K_i(X)^{(j)}.$$
\end{Cor}

\begin{D} \label{defn:V_X}Let $\mc C$ be a closed simplicial model category, and suppose that $\mc X$ is an object of $\mc C$. Given a fibrant replacement $\mc X \to \mc X'$, consider the functor $V_\mc X$ taking an object $X$ of $\mc C$ to the fundamental groupoid of $\fhom(X, \mc X')$. This is independent up to unique isomorphism of the choice of fibrant replacement by abstract nonsense.
We call $V_\mc X$ the associated category fibered in groupoids over $\mc C$. A 1- and 2-morphism of categories fibered in
groupoids over $\mc C$ is the standard one and we denote by $\Hom_f(V_\mc X, V_\mc Y)$ the set of 1-morphisms
$V_\mc X \to V_\mc Y$ strictly functorial with respect to pullback.  \end{D}
Very often, these groupoids have the structure of Picard categories and they form Picard categories fibered in $\mc C$. Recall from \cite{AAG}, 3.6 that a Picard category fibered over $\mc C$, $P$, is, for every object $X$ of $\mc C$, a Picard category $P_X$ and for every morphism $X \to Y$ an additive functor $P_Y \to P_X$ compatible with composition in the obvious sense. \\
The following proposition is formal, and is surely known in more generality:
\begin{Prop} \label{Prop:virtualhomotopy} [Pre-rigidity, proof of \cite{Riou}, Chapitre III, section 10] Let $T$ be as above and consider the
category of pointed or unpointed simplicial (pre-)sheaves on $T$. Suppose $\mc X$ and $\mc Y$ are
objects thereof, with $\mc X$ cofibrant and $\mc Y$ fibrant, with associated fibered categories in groupoids $V_\mc X, V_\mc Y$, and suppose that
\begin{itemize}
  \item ${\Hom}_{\mc H(\mc C)}(\mc X, \Omega \mc Y) = 0$.
  \item $\fhom(\mc X, \mc Y)$ is an $H$-group.
\end{itemize}
 Then we have a canonical map
$$\Hom_{\mc H(\mc C)}(\mc X, \mc Y) \to \Hom_{f}(V_\mc X, V_\mc Y)$$
which associates to an element of the left a functor of fibered
categories $\phi: V_\mc X \to V_\mc Y$, canonical up to unique
isomorphism.
\end{Prop}
\begin{proof}
If $\Phi$ is in ${\Hom}_{\mc H(\mc C)}(\mc X, \mc Y) = \pi_0(\fhom(\mc X, \mc Y))$, it induces for any $X \in \mc C$ a map $\phi_X: \mc V_\mc X(X) \to \mc V_\mc Y(X)$,
functorial in $X$, by choice of a representative $\phi$ of $\Phi$ in $\fhom(\mc X, \mc Y)$ . If
$\phi$ and $\phi'$ induce the same homotopy-class, there is a
homotopy $h: \Delta^1 \times \mc X \to \mc Y$ from $\phi$ to
$\phi'$ which gives an isomorphism
$\op{iso}_{h,X}: \phi_X \to \phi'_{X}$. Moreover, it is easy to see
that if there are two homotopies $h$ and $h'$ which are homotopic,
they induce the same isomorphism of functors. The obstruction for $\op{iso}_{h,X}$ to be canonical lies in the fundamental group of $\fhom(\mc X, \mc Y)$ which can be identified with ${\Hom}_{\mc H(\mc C)}(\mc X, \Omega \mc Y)$ which is 0 by assumption.
\end{proof}

\section{Algebraic stacks} \label{appendix:algebraicstacks}
In this section we recall the necessary facts about algebraic stacks that will be needed. It is neither self-contained nor complete, and we refer the reader to for example \cite{MB} or \cite{GIT} for more exhaustive treatments. We refer to Définition 3.1 and
Définition 4.1, \cite{MB} for the necessary definitions of algebraic stacks. In particular an $S$-stack is a sheaf in groupoids on $(Aff/S)_{et}$.
Furthermore, let $T$ be a full subsite of $(Aff/S)_{et}$ (\ie a full subcategory with a Grothendieck topology such that a cover in the former is one in the latter).
By abuse of language, we say that that a category fibered in groupoids over
$T$ is resp. a stack, an algebraic stack or Deligne-Mumford stack if
it is the restriction of a stack, algebraic stack or Deligne-Mumford
stack.


Again to fix notation we recall \cite{MB}, Application 14.3.4:
\begin{D} We say that a representable 1-morphism $F: X \to Y$ of algebraic stacks is projective (or quasi-projective) if there is a coherent locally free sheaf $\mc E$ on $Y$ and 2-commutative diagram $$\xymatrix{X \ar[r]^I \ar[dr] & \bb P(\mc E) \ar[d]^P \\ & Y}$$ with $I$ a (representable) closed immersion (resp. quasi-compact immersion) and $P$ is the canonical projection.\footnote{This is what many authors call a "strongly projective" (or "strongly quasi-projective") 1-morphism.}
\end{D}
By \cite{Thomason3} there are many examples of when an equivariant scheme which is projective as a scheme is also equivariantly projective.
Given an algebraic stack $\mc X$, an algebraic space $X$ and a fppf-morphism $X \to \mc X$, the associated groupoid $[X_1 \rightrightarrows X_0]$ is an fppf-presentation of $\mc X$ (cf. \cite{MB}, Corollaire 10.6). We now recast the above in a setting which will make it more natural to apply various auxiliary results, which is that of a simplicial setting.

\begin{D} Let $T$ be a site. The category of presheaves and
sheaves on this site is denoted by $pShv(T)$ and $Shv(T)$
respectively. \\\\ The category of simplicial objects of a category
$\mc C$, \ie functors $\Delta^{op} \to \mc C$, is denoted by
$\Delta^{op} \mc C$ or $s \mc C$.
\end{D}
Recall that whenever $T$ has enough points a morphism of simplicial presheaves in $T$ is said to be a local equivalence if
it induces weak equivalences of simplicial sets on all stalks. \\
Let $U \to X$ be a morphism of an object $X$ in
$T$. The nerve of this morphism is the simplicial object $\mc
N(U/X)$ whose $n$-simplices are given by the product $U \times_X U
\times_X U \ldots \times_X U$ ($n$ times). Given a presheaf of
simplicial sets on $T$ we have an associated cosimplicial functor
$\Delta \to \op{Set}$, $[n] \mapsto \mc F(\mc N(U/X)_n)$. The Cech
cohomology with respect to the covering $U \to X$ is the simplicial
set $$\mathbf{H}(U/X, \mc F) := \holim_{\Delta} \mc F(\mc N(U/X)_n).$$
We say that $\mc F$ satisfies descent if for any $X$ and any
covering $U \to X$ in $T$, the map $$\mc F(X) \to \mathbf{H}(U/X,
\mc F)$$ is a weak equivalence.

\begin{D} A presheaf $\mc F$ of simplicial sets on a site $T$ is said to be flabby, if
for any (and thus each) simplicially fibrant replacement $\mc F \to \mc F'$, and
any $X \in T$, the map  $$\mc F(X) \to \mc F'(X)$$ is a weak
equivalence of simplicial sets.
\end{D}

\begin{T} [\cite{Toen}, Theoreme 1.2] $\mc F$ is flabby if and only if it satisfies descent.
\end{T}
Thus any simplicially fibrant simplicial presheaf satisfies descent. It follows from the definition that a groupoid
is flabby if and only if it is a stack.\\
If $\mc X$ is an $S$-stack, there is sheaf of simplicial sets
defined as follows: Let $U$ be an object in $(Aff/S)$, and let
$\overline {\mc X}$ be the associated fibered category over $(Aff/S)$. The
category $F_{\mc X}(U) :=  Hom_{Cat/S}(\overline U, \mc X)$ is a
groupoid, and its nerve is a simplicial set $BF_{\mc X}$.

\begin{D} \label{defn:extyoneda} Let $T$ be a site, and consider the category of simplicial presheaves on $T$, $\Delta^{op}p\op{\mathbf{Shv}}(T)$. The full subcategory of simplicial sheaves is denoted by $\Delta^{op}\op{\mathbf{Shv}}(T)$. If $\mathfrak{Ch}$ is the category of stacks on $T$, we call the functor $B: \mathfrak{Ch}(T) \to \Delta^{op}\op{\mathbf{Shv}}(T)$ constructed above the extended Yoneda functor.
\end{D}

Furthermore, a (cartesian) quasi-coherent $\calo_{\mc X}$-module on an algebraic stack $\mc X$ viewed as a simplicial set is an assignment of a
quasi-coherent (resp. coherent, locally free, etc)  $\mc F_n$ on each $\mc X_n$ such that for any $\phi: [n] \to [m]$ we have an isomorphism $\phi^*: \phi^* \mc F_n \to \mc F_m$
compatible with compositions $[n] \to [n'] \to [n'']$. Coherent and locally free sheaves are defined analogously. \\
As an example (cf. \cite{HodgeIII}, 6.1.2), let $G$ be a group scheme, finitely presented, separated and
faithfully flat over a scheme $S$. Let $X$ be an algebraic space over $S$. We say that $G$ acts on $X$ if there is a
morphism $\mu:G \times_S X \to X$ satisfying the usual associativity and unit-constraints. If $\mc F$ is a $\calo_X$-module,
we say that $G$ acts on $\mc F$, or that $\mc F$ is $G$-equivariant,
if there is an isomorphism of $\calo_{G \times_S X}$-modules $$\phi:
\mu^* \mc F = p_2^* \mc F$$ satisfying the associativity
constraint, on $G \times_S G\times_S X$:
$$p_{23}^* (1 \times \mu)^* \phi = (\mu \times 1)^* \phi.$$ We employ the analogous definition for complexes of quasi-coherent $\calo_X$-modules. To an algebraic space  $X$ with a group action $G$, we can form the following
simplicial algebraic space:
$$\xymatrix{[X/G/S] := X \ar@<1ex>[r]  & \ar@<1ex>[l] \ar@<0ex>[l] G \times_S X \ar@<0.5ex>[r] \ar@<1.5ex>[r]& \ar@<1.5ex>[l] \ar@<0.5ex>[l] \ar@<2.5ex>[l] G \times_S G \times_S X \ldots}$$
Here the maps are either projection or multiplication-maps, and the non-written arrows in the other directions are given by repeated applications of the unit-map $e$. The above condition that
$\mc F$ is $G$-equivariant can equivalently be rephrased as that
$\mc F$ is the degree 0-part of a cartesian $\calo_{[X/G/S]}$-module on
$[X/G/S]$ with descent-data. \\

Yet another way of defining a quasi-coherent $\calo_{\mc X}$ on an algebraic stack $\mc X$, is in the following way: Given an algebraic space
$U$ and a 1-morphism with $U$ an algebraic space, $s: U \to \mc X$, we have an quasi-coherent $\calo_U$-module $\mc F_s$ on $U$. Given two 1-morphisms of algebraic spaces
$s: U \to \mc X, t: V \to \mc X$, a morphism $f: U \to V$, and a 2-isomorphism $h: t \circ f \Rrightarrow s$, an isomorphism
$$\phi_{f,t, h}: f^* \mc F_t \simeq \mc F_s.$$
Given morphisms of algebraic spaces $U \stackrel{f}\to V \stackrel{g}\to W$, and 1-morphisms $s: U \to \mc X, t: V \to \mc X, w: W \to \mc X$, and 2-isomorphisms
$h: t \circ f \Rrightarrow s$ and $j: w \circ g \Rrightarrow t$ an equality
$$\phi_{f,t,h} \circ f^* \phi_{g,w,j}= \phi_{f \circ g, w, h \circ j}.$$
Given two quasi-coherent $\calo_{\mc X}$-modules $\mc F$ and $\mc E$, a morphism between them is morphism $\mc F_s \to \mc E_s$ for every morphism $s: U \to \mc X$ with $U$ an algebraic space compatible with the
isomorphism $\phi$ in the obvious way.

\begin{D} The Quillen $K$-theory space of an algebraic stack $\mc X$, $K(\mc X)$ is defined to be the space $\Omega BQC$, with $C$ being the exact category of (coherent)
vector bundles on $X$. The $K$-theory groups $K_i(\mc X)$ are defined to be $\pi_i$ of the corresponding loops-space $\Omega BQC$. Similarly, one defines the $G$-theory space and $G$-theory of an algebraic stack
$\mc X$, $G_i(\mc X)$, as the corresponding object considering the category of coherent $\calo_\mc X$-modules instead.
\end{D}

The main standard properties of $K-$ and $G$-theory are summarized in the following theorem (compare with \cite{Toen}, Proposition 2.2, note however that it does not seem to be true that most of the results in this proposition automatically generalize from the case of schemes. Indeed, this is the main point of the article \cite{Thomason4} where the equivariant versions of non-cohomological $K$ and $G$-theory are studied):

\begin{T} \label{thm:K-propertiesstack} Fix a separated algebraic stack $\mc X$. Then we have

\begin{itemize}
  \item  $K(-)$ is contravariantly functorial with respect to 1-morphisms of algebraic stacks, and is covariantly functorial with respect
  to representable projective morphisms between algebraic stacks with the resolution property.
  \item $G(-)$ is covariantly functorial with respect to proper representable 1-morphisms.
  \item Let $\mc E$ be a vector bundle of rank $n$ on $X$, and consider the canonical bundle $\calo(1)$ on
  $\pi: Proj_X(Sym^\bullet \mc E) = \bb P(\mc E) \to \mc X$. Then we have a homotopy equivalence
  $$\bigvee_{j=0}^{n-1} K(X) \to K(\bb P(\mc E))$$
  induced by $(f_j)_{j=0}^{n-1} \mapsto \sum_{j=0}^{n-1} \pi^* f_j \otimes \calo(-j)$. Same formula holds for $G$.
  \item  Let $\mc E$ be a vector bundle on $\mc X$, and $T$ a torsor of $\mc E$ over $\mc X$. Then $G(\mc X) \to G(T)$ is a homotopy equivalence.

\end{itemize}
\end{T}
\begin{proof} The first result is proven as in $\cite{Thomason4}$, Theorem 3.1. and most of the results are proven using the classical techniques
or modifying the same using loc.cit. As we shall only need the above theorems in the special cases of their
associated virtual categories we will contend ourselves with the above statements without proofs.
\end{proof}

An additional object will enter onto our stage, $K$-cohomology, which in this form is borrowed from \cite{Toen}.

\begin{D} Let $T = Aff/S_{sm}$, the category of affine $S$-schemes with the smooth topology. Denote by $K^{TT}_\bb Q$ a $T$-simplicially
fibrant model of the simplicial presheaf on $T$ that represents rational Thomason algebraic $K$-theory and let $X$ be a simplicial $T$-sheaf.
The $K$-cohomology $K^{sm}$ is the simplicial presheaf (automatically flabby) $X \mapsto K^{sm}(X) := \fhom(X,K^{TT}_\bb Q)$. We define the $K$-cohomology groups $K^{sm}_i(X)$ to be $\pi_i(\fhom(X,K^{TT}_\bb Q))$. Also define $G^{sm}_\bb Q$ to be the $G$-cohomology of \cite{Toen}.
\end{D}

The definition of $K^{sm}(X)$ of \cite{Toen} is different, and exhibits $K^{sm}(\mc X)$ more properly as a $S^1$-spectrum. But by \ibid  \hbox{ } Proposition 2.2, the given spectrum is flabby when restricted to the small smooth site on the algebraic stack (\ie a smooth presentation is a cover) and equal to ordinary
(rational) Thomason $K$-theory for a regular Noetherian finite dimensional algebraic space or scheme. Because $\op{holim}$ preserves weak equivalences,
for a regular stack with smooth presentation $X \to \mc X$, we have weak equivalences $K^{sm}(\mc X) = \mathbf{H}(X/\mc X), K^{sm}) = \mathbf{H}(X/\mc X,K^{TT}_\bb Q) = \fhom(\mc X, K^{TT}_\bb Q)$ so Toen's $K$-cohomology necessarily coincides with our $K$-cohomology in this case. Also recall that for a scheme in addition to being finite dimensional Noetherian admit an ample family of line bundles $K^{TT}_\bb Q(X)$ represents rational Quillen $K$-theory.
By \cite{Thomason6}, Theorem 2.15 rational $G$-theory has étale descent for separated Noetherian schemes of finite Krull-dimension and thus rational $G$-theory has descent for algebraic spaces. It should be noted that Toen's corresponding $G$-cohomology theory does not have smooth descent in general so cannot be defined as values of an algebraic stack in some simplicial sheaf representing $G$-theory in $\bb A^1$-homotopy theory.
\\ By \cite{Toen}, Proposition 1.6, there is a natural transformation $K \to K^{TT}_\bb Q$ that
can be realized as, for a smooth presentation $X \to \mc X$, the augmentations $K(\mc X) \to \mathbf{H}(X/\mc X, K_\bb Q)$. \\
With these remarks it follows from Theorem \ref{thm:voevmorelGr} that we have the following proposition:
\begin{Prop} Let $T = \mathfrak R_{S, sm}$ be the category of regular $S$-schemes with the smooth topology. Then for any regular algebraic $S$-stack $\mc X$
there is an $\bb A^1$-weak equivalence $K^{TT}_\bb Q \to (\bb Z \times \op{Gr})_\bb Q$ so that
$$K^{sm}_i(\mc X) = \Hom_{\mc H(T)}(\mc X, R\Omega^i (\bb Z \times \op{Gr})_\bb Q).$$
\end{Prop}

\begin{Prop} [\cite{Toen}, Proposition 2.2] \label{Poincare-duality} The conclusions of Theorem $\ref{thm:K-propertiesstack}$ hold with $K$ (resp. $G$) replaced by $K^{sm}$ (resp. $G^{sm}$), at least whenever restricted to the category of regular stacks. Moreover, for a regular algebraic stack there is Poincaré duality; the natural map $K^{sm}(\mc X) \to G^{sm}(\mc X)$ is a weak equivalence.
\end{Prop}

\mps{DONE. fix this so that $G$-cohomology is really $G$-cohomology in the sense of Toen. The definition above is wrong. Toen shows $\underline{G}(X)=\underline{K}(X)$ for regular stacks, $\underline{K}$ has descent, and $\underline{G}(X)=G(X)$ for algebraic spaces, so for a regular stacks one obtains $\underline{K}(\mc X) = \mathbf{H}(X/\mc X, \underline{K}) = \mathbf{H}(X/\mc X, \underline{G}) = \mathbf{H}(X/\mc X, G) =
\mathbf{H}(X/\mc X, K)$. }

\bibliographystyle{amsalpha}
\bibliography{stuff}
\end{document}